\documentclass[aap,preprint]{imsart}
\usepackage[utf8x]{inputenc}
\usepackage{amsmath}
\usepackage{amsthm}
\usepackage{amssymb}
\usepackage{hyperref}
\usepackage{tikz}
\usepackage{amsfonts}
\usepackage{bm}
\usepackage[framemethod=TikZ]{mdframed}
\usepackage{todonotes}

\usepackage[capitalize]{cleveref}

\usetikzlibrary{positioning,arrows}
\usetikzlibrary{matrix}
\usetikzlibrary{shapes,snakes}
\usetikzlibrary{arrows,decorations.pathreplacing,patterns}

\mdfdefinestyle{Example}{%
    linecolor=blue,
    outerlinewidth=1pt,
    roundcorner=2pt,
    innertopmargin=\baselineskip,
    innerbottommargin=\baselineskip,
    innerrightmargin=20pt,
    innerleftmargin=20pt,
    backgroundcolor=gray!25!white}
    
\newtheorem{theorem}{Theorem}[section]
\newtheorem{lemma}[theorem]{Lemma}
\newtheorem{proposition}[theorem]{Proposition}

\newtheorem{definition}[theorem]{Definition}
\newtheorem{corollary}[theorem]{Corollary}
\newtheorem{question}[theorem]{Question}
\newtheorem*{theorem*}{Theorem}
\newtheorem*{corollary*}{Corollary}
\theoremstyle{remark}
\newtheorem{remark}[theorem]{Remark}
\newtheorem{example}[theorem]{Example}

\definecolor{darkblue}{rgb}{0.0,0,0.7} 

\newcommand{\defn}[1]{\emph{#1}}

\def\R{R}

\def\E{\mathbb{E}}
\def\P{\mathbb{P}}

\def\tO{\tilde{\mathcal{O}}}
\def\tTheta{\tilde{\Theta}}
\def\O{\mathcal{O}}

\def\a{\alpha}
\def\ka{\kappa}
\def\si{\sigma}
\def\bC{\bm{C}}
\def\ll{\bm{\ell}}
\DeclareMathOperator{\Var}{Var}
\DeclareMathOperator{\Cov}{Cov}
\DeclareMathOperator{\Occ}{Occ}
\DeclareMathOperator{\MSet}{MSet}
\DeclareMathOperator{\NoRn}{NoRn}

\def\bD{\bm{D}}
\def\Kv{K_1}
\def\Ka{K_2}
\def\Kab{K_3}
\def\Kb{K_4}
\def\Kbb{K_5}
\def\Kc{K_6}
\def\Kd{K_7}
\def\Kdb{K_8}
\def\Keb{K_{9}}
\def\Kec{K_{10}(\eps)}
\def\Ku{K_{11}}
\def\Kl{K_{12}}

\newcommand{\esper}{\mathbb{E}}
\newcommand{\MWST}[1]{\mathcal{M} \big( #1 \big)}
\newcommand{\compl}[1]{#1^c}
\newcommand{\word}{\sigma}
\newcommand{\const}{\mathbf{C}}
\newcommand{\ub}{\bm{u}}
\newcommand{\uu}{\bm{u}}
\newcommand{\ptn}{\mathcal{P}}
\newcommand{\Err}{\mathrm{Err}}
\def\({\left(}
\def\){\right)}

\newcommand{\pa}{\mathcal{A}}

\def\Dep{G}
\def\WDep{G}
\def\WDepInd{H}
\def\ellorr{r}
\def\eps{\varepsilon}
\def\One{\bm{1}}
\def\OccSP{\Occ_n^{\mathcal A}}
\def\OccMP{\Occ^{\tau}}
\def\arcs{a}

\begin{document}

\begin{frontmatter}
\title{Central limit theorems for patterns in multiset permutations and set partitions}
\runtitle{Central limit theorems for patterns}

\begin{aug}
  \author{\fnms{Valentin} \snm{F\'eray}
\thanksref{t1}\ead[label=e1]{valentin.feray@math.uzh.ch}}

\thankstext{t1}{The author's research is partially supported by the Swiss National Science Fundation (SNSF),
grant nb 200020\_172515.}
\runauthor{V. F\'eray}

\affiliation{University of Zurich}
\address{Institute of Mathematics, University of Zurich,\\ Winterthurerstrasse 190, 8057 Zürich, Switzerland \\
\printead{e1}}
\end{aug}

\begin{abstract} 
   We use the recently developed method of weighted dependency graphs
   to prove central limit theorems for the number of occurrences of any fixed pattern in multiset permutations and in set partitions.
   This generalizes results for patterns of size 2 in both settings, obtained by Canfield, Janson and Zeilberger
   and Chern, Diaconis, Kane and Rhoades, respectively.
\end{abstract}



\end{frontmatter}

\section{Introduction}

\subsection{Background and informal presentation of the main results}

A natural parameter of interest in the framework of random combinatorial structures
is the number of occurrences of a given substructure.
When this substructure has a fixed size,
we observe in many cases that this number is asymptotically normal.
Such a central limit theorem (CLT) for substructures was first proved
for random graphs \cite{RucinskiCLTSubgraphs,JansonOrthogonalDecomposition}
and random words \cite{flajolet06hidden}.
More recently, similar results for pattern occurrences
in uniform random permutations were obtained:
see the works of Fulman \cite{fulman04inversions} (for inversions and descents),
Goldstein \cite[see in particular Example 3.2]{goldstein05CombiCLT} (for consecutive patterns),
B\'ona \cite{BonaMonotonePatterns} (for monotone patterns, both in the consecutive and classical settings),
Janson, Nakamura and Zeilberger \cite{janson2015asymptotic} (for general classical patterns)
and Hofer \cite{hofer2018central} (for vincular patterns).
We also refer to the work of the author \cite{feray13Ewens} and
Crane, DeSlavo and Elizalde \cite{crane18Mallows} for results for Ewens distributed and Mallows distributed
random permutations, respectively.

Most of these results are based on the theory of dependency graphs,
used either in combination with cumulant estimates or Stein's method,
see {\em e.g.} \cite[Section 3]{hofer2018central} for an overview of these tools.
One exception is the work of Janson, Nakamura and Zeilberger \cite{janson2015asymptotic},
which uses the theory of $U$-statistics \cite{hoeffding48UStat,janson1997gaussian}.
In all these methods, a key feature is the independence of occurrences
of the given fixed substructure in disjoint sets of positions.
\medskip

In this paper, we investigate CLTs for substructures in two other families of combinatorial objects:
multiset permutations and set partitions.
For both objects, some notion of patterns have been studied in the literature,
see \cite{albert2001permutations} and \cite{diaconis14average}, respectively.
In both settings, a CLT is only known
for the simplest kind of patterns: inversions in multiset permutations,
where the central limit theorem was established by 
Canfield, Janson and Zeilberger~\cite{canfield11mahonian} (see also~\cite{thiel2016inversion})
and crossings in set partitions, from the work of 
Chern, Diaconis, Kane and Rhoades~\cite{CLT_SetPartitionsStatistics}.
The methods used in these papers do not seem to be generalizable to longer patterns.
Indeed they are based on the following facts, which only hold for inversions and crossings, respectively:
the explicit generating functions
of inversions in multiset permutations has a simple explicit product from;
conditionally on the starting and ending points of blocks                
in set partitions, the number of crossings is a sum of independent random variables.

What makes patterns in multiset permutations and set partitions harder to study is 
that occurrences of a given pattern in disjoint sets of places
are no longer independent events.
We overcome this difficulty by using the theory of {\em weighted dependency graphs},
recently developed by the author \cite{feray18dependency}.

Our main results, \cref{Thm:CLT_Multiset,Thm:CLT_SetPart}, are {\em the asymptotic normality of the number of occurrences
of any fixed pattern} in random multiset permutations and in random set partitions.
For multiset permutations, we need a slight regularity assumption on the sequence
of multisets that we consider.
We refer to \cref{ssec:ThmMultiset,ssec:ThmSetPart} respectively, for precise definitions of the notions
of patterns in both settings and for precise statements of the main results.

\subsection{Terminology and notation}
Before stating precisely our main theorem, we introduce some notation.

{\em Sets and multisets.} As usual, we write $[n]$ for the set $\{1,\ldots,n\}$.
A multiset is an {\em unordered} collection of elements, with possible repetitions.
Given a multiset $B$ we write $|B|$ and $\#B$
for its number of elements, and number of \defn{distinct} elements, respectively.

{\em Probability theory.}
Indicator functions will be denoted with the symbol $\One$, namely
\[\One[P]:=\begin{cases}
  1\text{ if $P$ holds,}\\
  0\text{ otherwise.}
             \end{cases}
\]
Throughout the paper, 
we say that a sequence $(X_n)_{n \ge 1}$ of real-valued
random variables is {\em asymptotically normal} if
the following convergence in distribution holds:
    \begin{equation}
      \tfrac{X_n- \esper X_n}{\sqrt{\Var X_n}} \to_d Z \text{ as }n \to \infty,
        \label{EqConclusionMainThm}
    \end{equation}
    where $Z$ is a standard Gaussian variable.
\smallskip

{\em Asymptotic notation.}    
Besides, we use the standard $\O,\Theta,\tO,\tTheta$ symbols for asymptotic comparisons:
we write $f(n)=\O(g(n))$ if there exists $C>0$ such that $|f(n)| \le C g(n)$, for all $n$ sufficiently large.
Furthermore, $f(n)=\Theta(g(n))$ stands for $f(n)=\O(g(n))$ and $g(n)=\O(f(n))$.
Finally, in the set partition section, we use the tilde variants, for bounds up to logarithmic factors:
more precisely $f(n)=\tO(g(n))$ means that there exists $d \ge0$ such that $f(n)=\O\big(g(n) \ln(n)^d\big)$;
and we write $f(n)=\tTheta(g(n))$ if $f(n) =\tO(g(n))$ and $g(n) =\tO(f(n))$.
We try and make explicit on which parameters the above constant $C$ (and exponent $d$) may depend or not.
Nevertheless, as a rule of thumb, it depends on the pattern (denoted $\tau$ or $\mathcal A$) we consider
and/or the order of the considered moment/cumulant,
but neither on the size $n$ of the objects (or the underlying multiset $M$ for multiset permutations), 
nor on the positions $i_1,\dots,i_\ell$ of a pattern occurrence.

\subsection{First main result: a CLT for patterns in multiset permutations}
\label{ssec:ThmMultiset}
Let $M$ be a finite multiset of positive integers.
Concretely, we can write $M=\{1^{a_1},2^{a_2},\ldots\}$, where exponents are used
to indicate multiplicities; let also $n=|M|=\sum_{j=1}^\infty a_j$.
A \defn{multiset permutation} (or permutation for short) of $M$
is a word containing exactly $a_j$ times the integer $j$ (for each $j \ge 1$).
Define $S_M$ as the set of multiset permutations of $M$. 
Naturally, we have $|S_M|=\frac{n!}{a_1!a_2!\cdots}$.
\begin{example}
The multiset $M=\{1^2,2^2,3\}$ has 30 permutations, one of which is $\word=23112$.
\end{example}
\noindent
We are interested in patterns in multiset permutations, following \cite{albert2001permutations}.
\begin{definition}
  Let $\tau$ be a permutation of size $k$.
  A multiset permutation $\word$ has an occurrence of $\tau$ in position $(i_1,\dots,i_\ell)$ ($i_1<\dots<i_\ell$)
  if the subsequence $\word_{i_1}\, \word_{i_2}\, \dots\, \word_{i_\ell}$ has distinct entries in the same relative order as $\tau$;
  formally if
  \[\word_{i_{\tau^{-1}(1)}} < \word_{i_{\tau^{-1}(2)}} < \dots < \word_{i_{\tau^{-1}(\ell)}}. \]
\end{definition}
For example,  the (multiset) permutation $\word=23112$ contains five occurrences of the pattern $\tau=21$:
in positions $(1,3)$, $(1,4)$, $(2,3)$, $(2,4)$ and $(2,5)$.

We are interested in the random variable $\OccMP_M:=\OccMP(\bm{\word})$, which gives the number of occurrences
of $\tau$ in a uniform random element $\bm{\word}$ of $S_M$.
Fixing $\tau$ and taking a sequence of multisets $M^{(m)}$,
we get a sequence of random variables $\OccMP_{M^{(m)}}$.
Our main theorem is a central limit theorem for $\OccMP_{M^{(m)}}$ 
under some regularity condition on the sequence $M^{(m)}$.
\begin{definition}
  Fix a positive integer $\ell$.
  A sequence $M^{(m)}$ is called $\ell$-regular,
  if there exists $K<1$ and $m_0$, such that, for $m \ge m_0$,
  the sum of the $\ell$ largest multiplicities in $M^{(m)}$
  is at most $K|M^{(m)}|$.
\end{definition}
\begin{theorem}
  \label{Thm:CLT_Multiset}
 Let $\tau$ be a pattern of size $\ell$ and 
 $(M^{(m)})_{m\geq 1}$ be a $\ell$-regular sequence of finite multisets.
 Then $\OccMP_{M^{(m)}}$ is asymptotically normal.
\end{theorem}
We also give estimates for the expectation and variance of $\OccMP_{M^{(m)}}$:
under the regularity hypothesis, the expectation is easily seen to be of order $\Theta(n^\ell)$.
Furthermore, it follows from our proof of \cref{Thm:CLT_Multiset} that the variance is $\O(n^{2\ell-1})$
and that this bound is tight up to subpolynomial factors, see \cref{ssec:discussion_variance_bounds}.

\begin{remark}
The above condition of $\ell$-regularity is not optimal,
as can already be observed for inversions ($\tau=21$).
In this case, our theorem gives the asymptotic normality under the condition that 
an asymptotically nonzero proportion of elements are different from the two most repeated parts.
This asymptotic normality is in fact known to hold under the weaker condition that
a non bounded number of elements are different from the single most repeated part \cite{canfield11mahonian}.
In general, our theorem could be improved if we could prove a lower bound for the variance
with less restrictive conditions, see \cref{question} below and the discussion after it.
\end{remark}

\subsection{Second main result: a CLT for patterns in set-partitions}
\label{ssec:ThmSetPart}
A \defn{set partition} $\pi$ of $[n]$ is a set of 
non-empty pairwise disjoint \defn{blocks} $B_1,B_2,\ldots,B_k$ with $\bigcup_{i=1}^kB_i=[n]$.
We will denote the set of set partitions of $[n]$ by $\mathcal P([n])$.
It is customary to represent graphically set partitions by a set of arcs
that join every pair of consecutive elements in the same block, see an example on \cref{fig:ExArcRepr}.
Formally, for $i,j\in[n]$, there is an \defn{arc} from $i$ to $j$ in $\pi$
if $j$ is the smallest number that is greater than $i$ and in the same block.

\begin{figure}[t]
  \begin{center}
    \includegraphics{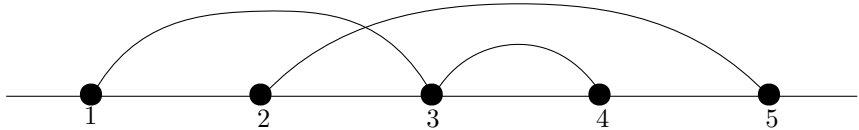}
  \end{center}
  \caption{The arc representation of the set partition $\pi=\{\{1,3,4\},\{2,5\}\}$}
  \label{fig:ExArcRepr}
\end{figure}

An \defn{arc pattern} of length $\ell$ is a subset $\pa\subseteq\{(i,j)\in [\ell]\times[\ell]:i<j\}$
such that for distinct elements $(i_1,j_1)$ and $(i_2,j_2)$ in $\pa$,
we have $i_1\neq i_2$ \defn{and} $j_1\neq j_2$.
Arc patterns of length $\ell$ are exactly the arc representations of set partitions of $\ell$,
but in the sequel, it is more natural to think of them as sets of arcs.
We say that the pattern $\pa$ occurs in a set partition $\pi$ in positions $x_1<x_2<\ldots<x_\ell$ 
if for every $(i,j)\in\pa$ there is an arc from $x_i$ to $x_j$ in $\pi$.
The number of occurrences of $\pa$ in a set partition $\pi$ will be denoted $\Occ^\pa(\pi)$.
As an example, an occurrence of the arc pattern $\{(1,3),(2,4)\}$ is a crossing 
in a set partition $\pi$; the set partition of \cref{fig:ExArcRepr} contains one such occurrence,
in positions $\{1,2,3,5\}$.

We are interested in the random variable $\OccSP:=\Occ^\pa(\bm \pi)$,
which gives the number of occurrences of a fixed arc pattern
in a uniform random set partition $\bm\pi$ of $[n]$.
The main result of this part is the following central limit theorem.
\begin{theorem}
  \label{Thm:CLT_SetPart}
 The number of occurrences $\OccSP$ of any fixed arc pattern $\pa$
 in a uniform random set partition of $[n]$ is asymptotically normal as $n\rightarrow\infty$.
 Moreover, asymptotically, we have $\mathbb E(\OccSP)=\tTheta(n^{\ell-a})$
 and $\Var(\OccSP)=\tTheta(n^{2\ell-2a-1})$, where $\ell$ and $a$ denote the length
 and the number of arcs of $\mathcal A$, respectively.
\end{theorem}

\begin{remark}
  Arc patterns are a particular case of the statistics called {\em patterns} in \cite{diaconis14average}.
  The latter are more general in the sense that we can require in an occurrence 
  that some $x_i$ is the first (or the last) element in its block
  or that $x_i$ and $x_{i+1}$ are consecutive
  for some $i$. Then the authors of \cite{diaconis14average} consider sums of the kind
  $\sum Q(x_1,\dots,x_\ell,n)$, where $Q$ is a polynomial, and the sum runs over occurrences of a given pattern $\pa$.
  This obviously generalizes the number of occurrences.

  The asymptotic normality of such statistics could also be investigated through weighted dependency graphs.
  Indeed, we think that with little extra effort (but heavier notation),
  we could include in the weighted dependency graph
  of \cref{ssec:StamWDG} some indicator variables 
  $F_i$ and $L_i$, indicating whether $i$ is the first or the last element in its block, respectively.
  Nevertheless, at this level of generality, it seems hard to find good variance bounds and estimates for the 
  parameters $R$ and $T_h$ of the weighted dependency graph.
\end{remark}

\subsection{Discussion on the proofs}
As said above, the proof relies on the theory of weighted dependency graphs.
There are two major difficulties in applying it.
\begin{itemize}
  \item The first one is to prove that the relevant random variables admit a suited weighted dependency graphs;
    this consists in bounding their joint cumulants.
    Thanks to the general theory of weighted dependency graphs, this can be reduced
    to bound the joint cumulants of simple indicator random variables.

    In the case of multiset permutations, since the corresponding joint moments are explicit,
    this is relatively easy.
    In the case of set partitions, we use a construction of a uniform random set partition
    through a urn model with a random number $M$ of urns, due to Stam \cite{stam1983RandomSetPartition}.
    We first bound the conditional cumulants with respect to $M$, and then use the law of total cumulance \cite{brillinger1969totalcumulance}.
    This needs deviation estimates on $M$ and a quite delicate analysis (\cref{sec:WDG_SetPartition}).
  \item The second difficulty is to find a lower bound on the variance.
    This is generally a delicate question, often left aside in the literature:
    in \cite{chatterjee08new}, the author writes
    ``to show that the bound [on Kolmogorov's distance] is useful,
    we require a lower bound on $\sigma^2$. We prefer
    to think of that as a separate problem.''
    
    In this paper, we provide such lower bounds on the variance for general patterns,
    both in multiset permutations and set partitions.
    In both cases, these bounds are based on the law of total variance
    (which is suited for finding lower bounds since it contains only nonnegative terms).
    For set partitions, we condition on the number $M$ of urns in the urn model
    and it turns out that one of the term in the law of total variance,
    namely the variance of the conditional expectation,
    is analyzable and already large enough to provide the needed lower bound
    (the analysis is however rather delicate, see \cref{sec:variance_estimate}).
    For multiset permutations, we condition on (one of) the position of the smallest element.
    This gives us a recursive inequality on the variance that can be analyzed to give the needed lower bound.
    This method has been recently introduced by Hofer \cite{hofer2018central},
    in the context of vincular patterns in (usual) permutations.
\end{itemize}

\subsection{Outline of the paper}
The paper is organized as follows. \cref{sec:WDG} introduces the necessary background on weighted dependency graphs,
in particular the normality criterion that we use.
\cref{part:multiset_perm,part:set_partitions} contain the proofs
of our main results
on multiset permutations and set partitions, respectively.
These sections are independent from each other.
\cref{sec:WDG_SetPartition,sec:variance_estimate} are devoted to technical proofs 
of the set partition section.
\bigskip

\section{Weighted dependency graphs}
\label{sec:WDG}
\subsection{Motivation and definition}
We first present the definition of usual dependency graphs
and explain informally why we need a weighted version of it.
\begin{definition}[Janson, \cite{JansonDependencyGraphs}]
    A graph $\Dep$ is a dependency graph for the family $\{Y_\a,\a \in A\}$
    of random variables
    if the two following conditions are satisfied:
    \begin{enumerate}
        \item the vertex set of $\Dep$ is $A$;
        \item if $A_1$ and $A_2$ are subsets of $A$ such
          that no edge connects a vertex of $A_1$ to one of $A_2$ in $\Dep$,
            then  $\{Y_\a,\a \in A_1\}$ and $\{Y_\a,\a \in A_2\}$ are independent.
    \end{enumerate}
\end{definition}
Roughly, a dependency graph encodes the dependency relations 
between the variables $Y_\a$:
variables not related by edges must be independent.

Consider now a sequence $S_n=\sum_{\a \in A_n} Y_{\a,n}$ of sum of random variables
and, for each $n$, a dependency graph $\Dep_n$ for the family $\{Y_{\a,n},\a \in A_n\}$.
If these dependency graphs are sparse enough, we might expect that 
$S_n$ behaves as a sum of independent variables and
is, under mild condition, asymptotically normal.
A precise normality criterion which formalizes this intuition, has been given by Janson
\cite[Theorem 2]{JansonDependencyGraphs} -- see also the work of Mikhailov \cite{MikhailovDependencyGraphs},
which has a higher range of applications.
\medskip

As explained in the introduction, dependency graphs are not suited to study patterns
in multiset permutations or set partitions,
since occurrences of a given pattern in disjoint sets of position
are in general dependent events.
Yet, such events are weakly correlated, 
and we will be able to use a weighted variant of dependency graphs,
that we now present.
The rest of this section follows \cite{feray18dependency}.
\bigskip

First we need something that quantifies the dependency between random variables.
To this purpose, we use mixed cumulants.
Cumulants have a long history and many different names:
they are referred to as Ursell functions (after Ursell \cite{ursell_1927})
or truncated/connected correlation functions in the physics literature,
as {\em ``déviation d'indépendence''} in a note of Schutzengberger \cite{Schutzenberger1947},
or as semi-invariants, for instance, in \cite{LeonovShiryaevCumulants} and \cite{JansonDependencyGraphs}.
We refer to \cite[Chapter 6]{JansonLuczakRucinski2000} for a modern presentation
of mixed cumulants and applications to random graph theory.

The (mixed) \defn{cumulant} of a family of random variables $X_1,X_2,\ldots,X_r$ 
defined on the same probability space and having finite moments
is defined as
\begin{equation}
  \kappa(X_1,X_2,\ldots,X_r):=\sum_{\pi\in\ptn([r])}\mu(\pi,{[r]})\prod_{B\in\pi}\E\left(\prod_{i\in B}X_i\right),
  \label{eq:def_cumulants}
\end{equation}
where $\ptn([r])$ is the lattice of set partitions of $[r]$ and $\mu$ is its M\"obius function.
If $X_i=X$ for all $i$ we abbreviate $\kappa(X_1,X_2,\ldots,X_r)$ as $\kappa_r(X)$.
Key properties of cumulants are the following:
\begin{enumerate}
  \item \label{item:ind_CumZero} If $\{X_1,X_2,\ldots,X_r\}$ can be written
    as a disjoint union of two independent non-empty sets of random variables then $\kappa(X_1,X_2,\ldots,X_r)=0$.
 \item a sequence $Y_n$ of random variables converges in distribution to a standard normal variable
   as soon as $\kappa_r(Y_n) \to \bm{1}[r=2]$ for all $r \ge 1$.
\end{enumerate}
If $\{Y_\a,\a \in A\}$ is a family of random variables with dependency graph $\Dep$,
then property \ref{item:ind_CumZero} above implies that, for indices $\a_1,\cdots,\a_r$
such that the induced graph $G[\a_1,\cdots,\a_r]$ is disconnected, we have
\[\kappa(Y_{\a_1},\ldots,Y_{\a_r}) =0.\]
For weighted dependency graphs, the idea behind the definition
is that the smaller the edge weights in the induced subgraph $G[\a_1,\cdots,\a_r]$ are,
the smaller the corresponding mixed cumulants should be.

In the sequel, a weighted graph is a graph with weights on its {\em edges}, belonging to $(0,1]$. 
Non-edges can be interpreted as edges of weight $0$, so that a weighted graph can be equivalently seen
as an assignment of weights in $[0,1]$ to the edges of the complete graph.
All our definitions are compatible with this convention.

For a weighted graph $H$, we define  $\MWST{H}$
to be the \defn{maximal weight of a spanning tree} of $H$,
the weight of a spanning tree being the product of the weights of its edges
(if $H$ is disconnected, there is no spanning tree
and as a consequence of the above convention, we have $\MWST{\WDep[B]}=0$).
The following definition was proposed in \cite{feray18dependency}.
\begin{definition}
\label{Def:WDG}
Let $\bC=(C_1,C_2,\cdots)$ be a sequence of positive real numbers.
Let $\Psi$ be a real-valued function on multisets of elements of $A$.

A weighted graph $\WDep$ is a $(\Psi,\bC)$ weighted dependency graph
for a family $\{Y_\a,\a \in A\}$ of random variables defined 
on the same probability space and having finite moments
if, 
for any multiset \hbox{$B=\{\a_1,\ldots,\a_r\}$} of elements of $A$,
one has
\begin{equation}
    \bigg| \ka\big( Y_\a ; \a \in B \big) \bigg| \le
    C_r \, \Psi(B) \, \MWST{\WDep[B]}. 
    \label{EqFundamental}
\end{equation}
\end{definition}
In examples of weighted dependency graphs, 
$\Psi$ and $\bC$ are simple or universal quantities,
so that the meaningful term is $\MWST{\WDep[B]}$.
Note that the smaller the weight on edges are,
the smaller $\MWST{\WDep[B]}$ is, which is consistent with intuition.

\subsection{A criterion for asymptotic normality}
Let $\WDep$ be a $(\Psi,\bC)$ weighted dependency graph 
for a family of variables $\{Y_\a, \a \in A\}$.
Let $I$ and $J$ be subsets of $A$.
If $I$ and $J$ have an element in common, we set $W(I,J)=1$.
Otherwise, we define
$W(I,J)$ as the maximal weight of an edge in $\WDep$
connecting an element of $I$ to an element of $J$
(if there is no such edge $W(I,J)=0$, which is consistent with the fact
that a nonedge can be replaced by an edge of weight $0$).

Finally, we introduce the following parameters ($h$ being a positive integer):
\begin{align}
    \R &= \sum_{\a \in A} \Psi(\{\a\});
    \label{EqDefR}\\
    \label{EqDefT}
    T_h &= \max_{\substack{\a_1,\ldots,\a_{h} \in A}} \left[ \sum_{\beta \in A} 
    W(\{\beta\},\{\a_1,\cdots,\a_h\}) 
    \frac{\Psi\big(\{\a_1,\cdots,\a_{h},\beta\}\big)}{\Psi\big(\{\a_1,\cdots,\a_{h} \}\big)}   \right].
\end{align}
Admittedly, the definition of $T_h$ is somewhat involved,
but these parameters turn out to be easy to estimate in practice.
In the particular case where $\Psi$ is the constant function equal to $1$,
$R$ is the number of variables, and each $T_h$
is within a factor $h$ of the maximal weighted degree of the graph
-- see \cite[Remark 4.9]{feray18dependency}.

Using these parameters, the following asymptotic normality criteria was given in 
\cite[Theorem 4.11]{feray18dependency}
\begin{theorem}
    \label{Thm:WDG}
    Suppose that, for each $n$, $\{Y_{n,i}, 1\le i \le N_n\}$ is a family of 
    random variables with finite moments defined on the same probability space.
    For each $n$, let 
    $\Psi_n$ a function on multisets of elements of $[N_n]$.
    We also fix a sequence $\bC=(C_r)_{r \ge 1}$, {\em not depending} on $n$.

    Assume that, for each $n$, one has a $(\Psi_n,\bC)$ weighted dependency graph $\WDep_n$\,
    for the family
    \hbox{$\{Y_{n,i}, 1\le i \le N_n\}$} and define the corresponding quantities
    $\R_n$, $T_{1,n}$, $T_{2,n}$, \ldots, by \cref{EqDefT,EqDefR}.
    
    Let $X_n = \sum_{i=1}^{N_n} Y_{n,i}$ and $\si_n^2= \Var(X_n)$.
    
    Assume that there exist numbers $\gamma_h$ and $Q_n$ and an integer $s\ge 3$ such that
    \begin{align}
        T_{h,n} &\le \gamma_h Q_n ;
        \label{EqBoundUnifT} \\
        \left(\tfrac{R_n}{Q_n}\right)^{1/s}\, \tfrac{Q_n}{\sigma_n} &\to 0\text{ as }n \to \infty.
        \label{EqHypoMainThm}
    \end{align}
    Then $X_n$ is asymptotically normal.
\end{theorem}
Note that establishing \eqref{EqHypoMainThm} requires a lower bound on $\sigma_n$,
which is often nontrivial to obtain.
This theorem is proved by bounding cumulants of $X_n$.
In particular, we get an upper bound on the variance.
\begin{proposition}\cite[Lemma 4.10]{feray18dependency}
  \label{prop:WDG_UpperBound_Var}
  We use the notation of  \cref{Thm:WDG}
  (but do not assume \eqref{EqBoundUnifT} and \eqref{EqHypoMainThm}).
  Then $\Var(X_n) \le 2 C_2 \R_n T_{1,n}$.
\end{proposition}
\begin{remark}
  \label{Rmk:WDGNonCstC}
  The above normality criterion can be adapted to sequences of random variables
  with a $(\Psi_n,\bC_n)$ weighted dependency graph $\WDep_n$,
  where $\bC_n$ depends on $n$.
  In such situation, we write $\bC_n=(C_{r,n})_{r \ge 1}$.
  The condition \eqref{EqHypoMainThm} should then be replaced by
  the fact that, for some $s \ge 3$ and all $r \ge 1$,
  the quantity $\left(\tfrac{R_n}{Q_n}\right)^{1/s}\, \tfrac{Q_n}{\sigma_n}$
  tends to $0$ faster than any power of $C_{r,n}$.
  The proof of this extension is a straightforward adaptation of
  the proof of \cite[Theorem 4.11]{feray18dependency}.
  We will use it in \cref{part:set_partitions} below.
\end{remark}

\subsection{Power of weighted dependency graphs}
An important property of weighted dependency graphs is the following stability property.
Consider a family of random variables $\{Y_\a,\a \in A\}$ 
with a $(\Psi,\bC)$ weighted dependency graph $\WDep$
and fix some integer $d \ge 1$.
We are interested in monomials $\bm Y_I:=\prod_{\a_i \in I} Y_{\a_i}$ of degree at most $d$,
i.e. $I$ is a multiset of elements of $A$ of size at most $d$ (counting repetitions),
which we will denote as $I \in \MSet_{\le d}(A)$.
This new family of random variables $\{\bm Y_I, I \in \MSet_{\le d}(A)\}$
admits a natural weighted dependency graph inherited from that of $\{Y_\a,\a \in A\}$.

To state this formally, we need to introduce some more notation.
The $d$-th power $\WDep^d$ of $\WDep$ is the weighted graph with vertex set $\MSet_{\le d}(A)$ 
and having an edge of weight $W(I,J)$ between every pair of vertices $(I,J)$ in $\MSet_{\le d}(A)$ 
(here, we describe the weighted graph $\WDep^d$ as a complete graph with possibly zero edge-weights).

Finally, a function $\Psi$ defined on multiset of $A$
is naturally seen as a function on multiset of multisets of $A$ by setting:
\begin{equation}
     \bm{\Psi}(\{I_1,\cdots,I_r\}) = \Psi(I_1 \uplus \cdots \uplus I_r)
\end{equation}
\begin{proposition}\cite[Proposition 5.11]{feray18dependency}
  \label{prop:WDG_Products}
  Consider random variables $\{Y_\a,\a \in A\}$
  with a $(\Psi,\bC)$ weighted dependency graph $\WDep$
  and fix some integer $d \ge 1$.
  Then, with the above notation,
  $\WDep^d$ is a $(\bm{\Psi},\bD_d)$ weighted dependency graph for the family
  $\{Y_I,\, I \in \MSet_{\le d}(A)\}$,
  where the constants $\bD_d=(D_{d,r})_{r \ge 1}$ depend only on
  $d$, $r$ and $\bC$.
\end{proposition}
In applications, the above proposition is used as follows.
We first find a weighted dependency graphs for some simple family of random variables:
typically indicators of basic events, such as $\sigma(i)=j$ for (multiset) permutations,
or the presence of an arc between given points $i$ and $j$ in set partitions.
The above theorem gives us automatically a weighted dependency graph for more complicated random variables,
such as indicators of having a fixed pattern in some given set of positions.
Then the normality criterion (\cref{Thm:WDG}) gives a central limit theorem for the number of occurrences
of this fixed pattern.

\section{Permutation patterns in multiset permutations}
In this section, we prove \cref{Thm:CLT_Multiset}.
We refer to \cref{ssec:ThmMultiset} for notation.

\label{part:multiset_perm}
\subsection{The weighted dependency graph}
We consider the following random variable $X_i^j$ on $S_M$, defined by:
\[X_i^j(\word)=\begin{cases}
            1\text{ if }\word_i=j;\\
            0\text{ otherwise}.
           \end{cases}
           \]
Let us also denote $A^M:=\{X_i^j:1\leq i\leq n, j\in M\}$.
The purpose of this subsection is to prove the following.
\begin{theorem}\label{thm:WDG_MSet}
 Consider the weighted complete graph $\WDep^M$
 on vertex-set $A^M$ with weights
 \[w(X_{i}^{j},X_{i'}^{j'})=\begin{cases}
                                   1\text{ if }i=i'\\
                                   1/a_j\text{ if }i\neq i'\text{ and }j=j'\\
                                   1/n\text{ otherwise}.
                                  \end{cases}
 \]
 Then $\WDep^M$ is a $(\Psi,\const)$ weighted dependency graph
 for the family $A^M$,
 where $\Psi$ is the function on multisets $B$ of elements in $A^M$ defined by
 \[\Psi(B)=\prod_{X_i^j\in B\text{ distinct}} \esper[X_i^j]
 =\prod_{X_i^j\in B\text{ distinct}}\frac{a_j}{n}\]
 and $\const=(C_1,C_2,\ldots)$ is a universal sequence of constants not depending on $M$.
\end{theorem}
\begin{proof}
The proof is relatively easy using the tools given in \cite{feray18dependency},
but these tools require to introduce some terminology/notation.
As a start, recall that if $\WDep$ is a weighted graph and $B$
a subset of its vertices, we denote $\WDep[B]$ the subgraph induced by $G$ in $B$
(if $B$ is a multiset, we see it as a set, by simply forgetting repetitions).
Also, for a weighted graph $H$,
$\MWST{H}$ is the maximal weight of a spanning tree of $H$.

We have to prove that 
for any multiset \hbox{$B=\{X_{i_1}^{j_1},\ldots,X_{i_r}^{j_r}\}$} of elements of $A^M$,
one has
\begin{equation}
  \bigg| \ka\big( X_{i_1}^{j_1},\ldots,X_{i_r}^{j_r} \big) \bigg| \le
    C_r \, \Psi(B) \, \MWST{\WDep^M[B]}. 
    \label{EqFundamentalRepeated}
\end{equation}
For a multiset $B$, we denote $B_1$, $B_2$, \dots, the vertex-sets of the connected components
of $\WDep^M_1[B]$, the graph obtaining from $\WDep^M[B]$ by keeping only edges of weight $1$.
Using \cite[Proposition 5.2]{feray18dependency},
it is equivalent to prove \eqref{EqFundamentalRepeated} or the following:
for any multiset $B$ of elements of $A^M$ with $|B|=r$ we have
\begin{equation}
  \left|\kappa\left(\prod_{X_i^j\in B_1}X_i^j,\prod_{X_i^j\in B_2}X_i^j,\ldots\right)\right|
  \leq D_r\Psi(B)\mathcal{M}\left(\WDep^M[B]\right),
  \label{eq:ToProve}
\end{equation}
for some sequence $D_r$ also independent of $M$.

By definition of $\WDep^M$,
vertices $X_i^j$ and $X_{i'}^{j'}$ are connected in $\WDep^M_1[B]$
if and only if $i=i'$ {\bf or} both $j=j'$ and $a_j=1$.
In this case, $X_i^jX_{i'}^{j'}=0$ a.s. unless $i=i'$ {\bf and} $j=j'$.
Of course, if one of the product on the left-hand side of \eqref{eq:ToProve} is a.s. $0$,
then the inequality is trivial.
Thus it suffices to consider the case where each component $B_k$
contains at most one distinct element $X_i^j$, say with multiplicity $m$.
But since $X_i^j$ is a Bernoulli random variable, we have $(X_i^j)^m=X_i^j$;
the right-hand side of \eqref{eq:ToProve} is also insensitive
to repetitions in $B$, so that we can assume $m=1$ (in each of the components $B_k$).
In other words, we only need to prove \eqref{eq:ToProve} in the case
the left-hand-side is $\big|\ka\big( X_{i_1}^{j_1},\ldots,X_{i_r}^{j_r} \big)\big|$,
for \defn{distinct} $i_1,\dots,i_r$ (and such that repeated entries $j$ in
the list $j_1,\cdots,j_r$ fulfills $a_j >1$, but we will not need this extra condition).
\bigskip

The proof of \eqref{EqFundamentalRepeated} in the all $i$ 
distinct case is based on the formula for joint moments in this case:
if $C$ is such a subset of $A^M$, we have
\begin{equation}
  \E\left(\prod_{X_i^j\in C}X_i^j\right)=
  \frac{\binom{n-|C|}{a_1-\#\{X_i^1\in C\},a_2-\#\{X_i^2\in C\},\ldots}}{\binom{n}{a_1,a_2,\ldots}}.
  \label{eq:joint_moments_Mset}
\end{equation}
Indeed, the numerator counts the number of multiset permutations $\sigma$ of $M$ with $\sigma(i)=j$ for all $X_i^j\in C$,
while the denominator is the total number of multiset permutations of $M$.
To get bounds on cumulants, we use this expression 
and the quasi-factorization technique, as developed in \cite[Section 5.2]{feray18dependency}.

%
Consider a family $\ub=(u_{\Delta})_{\Delta\subseteq[\ellorr]}$ of real numbers 
indexed by subsets of $[\ellorr]$ with $u_{\emptyset} \ne 0$.
We furthermore assume that $u_\delta=0$ implies that $u_\Delta=0$ as well 
for all subsets $\Delta$ containing $\delta$;
we call this \defn{the vanishing ideal condition}.
In the following, all families under consideration fulfill the vanishing ideal condition.
For such a $\uu$, we define
\[ \P_\Delta(\uu)
=\prod_{\delta \subseteq \Delta} \left( u_\delta \right)^{(-1)^{|\Delta|-|\delta|}}.\]
By convention, if the above fraction is $0/0$, we set $\P_\Delta(\uu)=0$.
A simple M\"obius inversion gives back
\[ u_\Delta= \prod_{\delta \subseteq \Delta} P_\delta(\uu). \]

Let $\WDepInd$ be a weighted graph on $[\ellorr]$.
We say that $\uu$ has the $\WDepInd$-quasi factorization property if,
for each $\Delta \subseteq [\ellorr]$ {\em of size at least $2$}, we have
\begin{equation}
  P_\Delta(\uu) = 1 +\O\big(\MWST{\WDepInd[\Delta]} \big).
  \label{eq:def_QF}
\end{equation}
The constants in the $\O$ symbol should depend only on $\ellorr$,
in particular, in the following, they are independent of $B$ and $M$.

Fix now a subset $B=\{X_{i_1}^{j_1},\cdots,X_{i_r}^{j_r}\}$
of $A^M$ with distinct $i_1,\cdots,i_r$.
For $\Delta \subseteq [r]$, we define
$\uu_\Delta= \esper\big[ \prod_{t \in \Delta}  X_{i_t}^{j_t} \big]$. 
(Note that the dependence in $M$ and $B$ is kept implicit here.)
From \cref{eq:joint_moments_Mset}, this is explicitly given as
\[ \uu_\Delta= \uu_\Delta^{(1)} \uu_\Delta^{(2)} \prod_{j \ge 1} \uu_\Delta^{(3,j)},\]
with 
\[\uu_\Delta^{(1)}=\frac{1}{\binom{n}{a_1,a_2,\ldots}};\quad \uu_\Delta^{(2)}=(n-|\Delta|)!;
\quad u_{\Delta}^{(3,j)} =\frac{1}{(a_j-\#\{t \in \Delta: j_t=j\})!}. \]
By convention, $u_{\Delta}^{(3,j)}=0$ if $a_j-\#\{t \in \Delta: j_t=j\}<0$.
Note that all these families have the vanishing ideal property.
We discuss the quasi-factorization property of each factor separately.
  \begin{enumerate}
    \item The first factor $u_{\Delta}^{(1)}$ is independent of $\Delta$
      and therefore $\uu^{(1)}$ trivially satisfies the $\WDepInd^{(1)}$-quasi
      factorization property, where
      $\WDepInd^{(1)}$ is the graph on $[r]$ with no edges.
    \item The family $\uu^{(2)}$ defined by $u_{\Delta}^{(2)}=(n-|\Delta|)!$
      satisfies the $\WDepInd^{(2)}$-quasi factorization property,
      where $\WDepInd^{(2)}$ is the complete graph on $[r]$
      with weight $1/n$ on each edge
      \cite[Proposition 5.10]{feray18dependency}.
    \item Fix $j \ge 1$ and consider the factor $u_{\Delta}^{(3,j)}=(a_j-\#\{i:X_i^j\in \Delta\})!^{-1}$.
      Let $\WDepInd^{(3,j)}$ be the graph with vertex-set $[r]$,
      and an edge of weight $\tfrac{1}{a_j}$ between $s$ and $t$ if $j_s=j_t=j$
      (in particular, if $j_s \ne j$, then $s$ is isolated in $\WDepInd^{(3,j)}$).
      We claim that $\uu^{(3,j)}$ has the $\WDepInd^{(3,j)}$ quasi factorization property.
      Indeed, if $\Delta$ contains an $s$ such that $j_s \ne j$,
      then $u_{\delta \cup \{s\}}^{(3,j)}=u_{\delta}^{(3,j)}$
      for $\delta \subseteq (\Delta \setminus \{s\})$
      and this implies $P_\Delta(\uu^{(3,j)})=1$ so that \eqref{eq:def_QF} holds trivially.
      On the other hand, if $\Delta \subseteq \{s: j_s = j\}$,
      then $u_{\delta}^{(3,j)}= (a_j-|\delta|)!^{-1}$ and
      the estimate \eqref{eq:def_QF} again 
      follows from \cite[Proposition 5.10]{feray18dependency}.
  \end{enumerate}
  Denote $\WDepInd$ the following graph with vertex-set $[r]$:
  the weight of the edge between $s$ and $t$ is the {\bf maximum} of the
  weights of the corresponding edges in $\WDepInd^{(1)}$, $\WDepInd^{(2)}$
  and in all the $\WDepInd^{(3,j)}$.
  Note that $\WDepInd$ corresponds to $\WDep^M[B]$,
  where $\WDep^M$ is defined in the statement of \cref{thm:WDG_MSet}. 
  Observe that, for any $\Delta$, we have 
  \[P_\Delta(\uu)=P_\Delta(\uu^{(1)}) \, P_\Delta(\uu^{(2)})\, 
  \prod_{j \ge 1} P_\Delta(\uu^{(3,j)}),\]
  where at most one factor in the infinite product is different from $1$.
  Together with the above observations, this implies that
  $\uu$ has the $\WDepInd$ quasi factorization property.
  In \cite{feray18dependency}, it is proved that the $H$ quasi-factorization property
  implies the so-called $H$ small cumulant property\footnote{See
  \cite[Proposition 5.8]{feray18dependency}; the implication is proved for families $\uu$ with non-zero entries,
  but the proof extends readily to families with the vanishing ideal property.},
  {\it i.e.} in particular the following inequality:
  \[ \big|\kappa_r(X_i^j,\, X_i^j \in B) \big| =
  \O \left( \MWST{\WDepInd} \, \prod_{X_i^j \in B} \esper(X_i^j)  \right),\]
  which is what we needed to prove.
\end{proof}

\subsection{A preliminary estimate}
Let $M=\{1^{a_1},2^{a_2},\dots\}$ be a finite multiset 
of size $n=\sum_{i\geq1} a_i$.
We let $(b_j)_{j \ge 1}$ be the non increasing reordering of $(a_i)_{i \ge 1}$
and define, for $j \ge 1$,
\begin{equation}
  n_{(j)}=n-b_1- \cdots - b_j.
  \label{eq:DefNj}
\end{equation}
Furthermore, for an integer $d \ge 1$,
we denote by $e_d(M)$ the $d$-th elementary symmetric function
evaluated in the numbers $a_1,a_2,\dots$, that is
\[e_d(M) := \sum_{j_1 < \dots < j_d} a_{j_1} \dots a_{j_d}.\]
These quantities turn out to be omnipresent when we evaluate the various
parameters needed to prove our central limit theorem.
\begin{lemma}
  \label{lem:Bound_Elem}
  For any $d \ge 1$ and any multiset $M=\{1^{a_1},2^{a_2},\dots\}$,
  we have
  \[\frac{1}{d!} n \, n_{(1)} \, \cdots \, n_{(d-1)}
  \le e_d(M) \le n \, n_{(1)} \, \cdots \, n_{(d-1)}.\]
\end{lemma}
In practice, the degree $d$ will be fixed, while $n$
tends to infinity, so that the above lemma gives us the exact order of magnitude
of $e_d(M)$.
For a $d-1$-regular sequence of multiset partition,
we clearly have $n_j=\Theta(n)$ for any $j \le d-1$, so that
$e_d(M)=\Theta(n^d)$.
\begin{proof}
  We start with the upper bound.
  Assume, without loss of generality that 
  the sequence $(a_i)_{i \ge 1}$ is nonincreasing,
  in which case $n_{(h)}=n-a_1- \dots-a_h$.
  Then, we have
  \begin{multline*}
    e_d(a_1,a_2,\cdots) = \sum_{j_1 < \dots < j_d} a_{j_1} \dots a_{j_d} \\
  \le \left( \sum_{j_1 \ge 1} a_{j_1} \right) 
  \, \left( \sum_{j_2 \ge 2} a_{j_2} \right)
  \, \cdots \, \left( \sum_{j_d \ge d} a_{j_d} \right)
  = n \, n_{(1)} \, \cdots \, n_{(d-1)}.
  \end{multline*}
  For the lower bound, we first observe that
  \[d! \, e_d(a_1,a_2,\cdots) = \sum_{j_1 \ge 1} a_{j_1} \left( 
  \sum_{j_2 \ne j_1} a_{j_2} \left( \dots \left( \sum_{j_d \ne j_1,\dots,j_{d-1}} a_{j_d} \right) \right) \right)\]
  The inner sum over $j_d$ can be bounded below as follows:
  \[\left( \sum_{j_d \ne j_1,\dots,j_{d-1}} a_{j_d} \right) 
  = n - a_{j_1} - \dots - a_{j_{d-1}} \ge n_{(d-1)}. \]
  This bound does not depend on $j_1,\dots,j_{d-1}$ and can therefore be factorized out of all sums.
  We then bound the sum over $j_{d-1}$ by $n_{(d-2)}$, which can also be factorized out.
  Iterating this procedure, we get
  \[d! \, e_d(a_1,a_2,\cdots)  \ge n_{(d-1)} n_{(d-2)} \dots n_{(1)} n,\]
  as claimed.
\end{proof}

\subsection{The central limit theorem}
\label{ssec:proof_MSet}
The goal of this section is to prove \cref{Thm:CLT_Multiset}.
Fix a pattern $\tau$ and a sequence of finite multisets,
where $M^{(m)}=\{1^{a_1^{(m)}},2^{a_2^{(m)}},\ldots\}$.
Most of the time we will omit the superscript $m$ and 
denote $n=|M^{(m)}|$.
We first observe that the number $\OccMP$ of occurrences
of a pattern $\tau$ in a uniform random element $\bm{\word}$ in $S_M$
can be written as 
\begin{equation}
  \OccMP(\bm{\word}) = \sum_{i_1<\dots<i_\ell \le n \atop j_1<\dots<j_\ell}
  X_{i_1}^{j_{\tau(1)}} \dots X_{i_\ell}^{j_{\tau(\ell)}}.
  \label{eq:OccPiGoodDecomposition}
\end{equation}
Combining \cref{thm:WDG_MSet,prop:WDG_Products},
we know that the family 
\[A_\tau:=\big\{ X_{i_1}^{j_{\tau(1)}} \dots X_{i_\ell}^{j_{\tau(\ell)}},
\quad i_1<\dots<i_\ell\le n,\, j_1<\dots<j_\ell \big\} \]
admits $(G^M)^\ell$ as $(\bD,\bm\Psi)$ dependency graph, where $\bm D$ and $\bm\Psi$ are as follows:
\begin{itemize}
  \item With a multiset $B$ of monomials in $A_\tau$,
    the function $\bm\Psi$ associates $\prod_{X_i^j} \frac{a_j}{n}$,
    where the product runs over the distinct variables $X_i^j$ appearing in some monomial in $B$;
  \item $\bD=(D_1,D_2,\dots)$ is a sequence of constants independent of the set partition $M$
    under consideration.
\end{itemize}
Our goal is to apply the normality criterion of \cref{Thm:WDG} to this (sequence of) dependency graph(s).
The first task is to estimate the parameters $R$ and $T_h$.
For $R$, this is immediate; indeed, we write
\[R= \sum_{i_1<\dots<i_\ell\le n,\atop j_1<\dots<j_\ell} \prod_{t=1}^\ell \frac{a_{j_t}}{n} 
= \binom{n}{\ell} \frac{e_\ell(M)}{n^\ell} = \Theta(e_\ell(M)). \]
We now consider $T_h$. By definition,
\begin{equation}
  T_h:=\max_{\alpha_1,\ldots,\alpha_h \in A_\tau}
  \left[\sum_{\beta\in A_\tau}W(\{\beta\},\{\alpha_1,\alpha_2,\ldots,\alpha_h\})\frac{\mathbf{\Psi}(\{\alpha_1,\alpha_2,\ldots,\alpha_h,\beta\})}{\mathbf{\Psi}(\{\alpha_1,\alpha_2,\ldots,\alpha_h\})}\right].
  \label{eq:Th_Multiset}
\end{equation}
We fix a set 
  $S=\{\alpha_1,\alpha_2,\ldots,\alpha_l\}\subseteq A_\tau$ and write 
  \[F(\beta):=W(\{\beta\},\{\alpha_1,\alpha_2,\ldots,\alpha_h\})\frac{\mathbf{\Psi}(\{\alpha_1,\alpha_2,\ldots,\alpha_h,\beta\})}{\mathbf{\Psi}(\{\alpha_1,\alpha_2,\ldots,\alpha_h\})}.\]
  We use the $\O$-notation with implicit constants depending on $h$. 
  We split the sum over $\beta$ into three parts:
  \begin{enumerate}
    \item Consider first the monomial $\beta=X_{i_1}^{j_{\tau(1)}} \dots X_{i_\ell}^{j_{\tau(\ell)}}$,
      which do not share any index with any of the $\alpha$ in $S$ (i.e. neither an upper,
      nor a lower index).
      For such $\beta$, the $W$ factor in $F(\beta)$ is $1/n$,
      while the quotient of $\Psi$ is $\prod_t \tfrac{a_{j_t}}{n}$,
      which yields
      \[F(\beta)=n^{-\ell-1} \prod_{t=1}^\ell a_{j_t}. \]
      We should sum this over ordered $\ell$-uplets $(i_1,\dots,i_\ell)$
      and $(j_1,\dots,j_\ell)$ with values in $[n]$.
      The sum over $(j_1,\dots,j_\ell)$ gives $n^{-\ell-1} e_\ell(M)$.
      This is independent of $i_1,\dots,i_\ell$ so that summing over
      the $\O(n^\ell)$ possible $\ell$-uplets $(i_1,\dots,i_\ell)$
      gives $\O\big(n^{-1} e_\ell(M)\big)$.

      Summing up, the total contribution of monomials $\beta$ not sharing any index
      with any $\alpha$ in $S$ to the sum in \eqref{eq:Th_Multiset}
      is $\O\big(n^{-1} e_\ell(M)\big)$.

    \item We now consider monomials $\beta$
      which do not share a lower index with any $\alpha\in S$
      but do share some upper index, say $j_r$, with some $\alpha\in S$.
      For such $\beta$, we have that
      $W(\{\beta\},\{\alpha_1,\alpha_2,\ldots,\alpha_h\})=\frac{1}{\alpha_{j_r}}$ and
      \[F(\beta) = \tfrac{1}{a_{j_r}} \prod_{t=1}^\ell \tfrac{a_{j_t}}{n}
      = n^{-\ell} \prod_{t \ne r} a_{j_t}.\]
      Again, we should sum over ordered $\ell$-uplets $(i_1,\dots,i_\ell)$  
      and $(j_1,\dots,j_\ell)$ with values in $[n]$.
      But the number of possible values of $j_r$ is finite since it must be chosen among
      the lower indices of $\alpha_1,\dots,\alpha_\ell$.
      Up to a constant factor, we can therefore only consider the sum 
      over $j_1,\dots,j_{r-1},j_{r+1},\dots,j_\ell$, which gives $\O\big(n^{-\ell}e_{\ell-1}(M)\big)$.
      Finally summing over the $\O(n^\ell)$ possible $\ell$-uplets $(i_1,\dots,i_\ell)$
      yields $\O\big(e_{\ell-1}(M)\big)$.
      This is the total contribution to the sum in \eqref{eq:Th_Multiset}
      of monomial $\beta$ in this second case.

    \item The third and last case is that of monomials $\beta$ 
      sharing some lower index, say $i_r$ with some $\alpha\in S$.
      In that case, the $W$ factor in $F(\beta)$ is one.
      Define 
      \[T(\beta)=\big\{t \in [\ell]:\, X_{i_t}^{j_{\tau(t)}} 
      \text{ is {\bf not} a factor of some $\alpha$ in }S\big\} \subseteq [\ell].\]
      Then we have $F(\beta)=\prod_{t \in T(\beta)} \frac{a_{j_{\tau(t)}}}{n}=
      n^{-|T(\beta)|} \prod_{t \in T(\beta)} a_{j_{\tau(t)}}$.
      The number of possible values for the set $T(\beta)$ is finite,
      so that it is enough to bound the sum in \eqref{eq:Th_Multiset}
      over $\beta$'s with a given value of $T(\beta)$.

      Given $T_0 \subseteq [\ell]$, a monomial $\beta$ with $T(\beta)=T_0$
      is described by the lists $(i_t)_{t \in T_0 \setminus [r]}$ and $(j_{\tau(t)})_{t \in T_0}$,
      and some additional finite choices (the values of $i_t$ and $j_{\tau(t)}$, for $t \notin T_0$,
      as well as that of $i_r$ if $r \in T_0$).
      Similarly as above, summing $F(\beta)$ over $(j_{\tau(t)})_{t \in T_0}$
      gives $n^{-|T_0|} e_{|T_0|}(M)$, while summing over $(i_t)_{t \in T_0}$
      yields $n^{|T_0|-1}$ or $n^{|T_0|}$, depending on whether $r$ is in $T_0$ or not.
      Therefore, the total contribution to the sum in \eqref{eq:Th_Multiset} 
      of monomials $\beta$ in this third case with $T(\beta)=T_0$
      is either $\O(n^{-1} e_d(M))$ for $d:=|T_0|\le \ell$ (in the case $r \in T_0$)
      or $\O( e_d(M))$ for $d:=|T_0|\le \ell-1$ (in the case $r \notin T_0$).
  \end{enumerate}
  From \cref{lem:Bound_Elem}, we see that, as long as $n_{(\ell-1)}$ tends to infinity
  (so in particular for an $\ell$-regular sequences of multiset partition $M^{(m)}$),
  the biggest of the above bounds is $\O\big(e_{\ell-1}(M)\big)$.
  We therefore conclude that $T_h = \O(e_{\ell-1}(M))$.
 \bigskip

Next, we need a lower bound on the variance.
Here we need our regularity assumption on the sequence of multiset partitions.
\begin{proposition}
  Let $\tau$ be a fixed pattern of size $\ell$ and consider a $\ell$-regular sequence $M^{(m)}$ of multisets.
  There exists a constant $\Kv >0$ such that
  \[\Var(\OccMP_{M^{(m)}}) \ge 
  \Kv |M^{(m)}|^{2\ell-3/2}.\]
  \label{prop:LowerBoundVarMP}
\end{proposition}
The proof of this statement being rather technical, 
we postpone it to the end of the multiset partition part,
i.e. \cref{sec:VarianceMP}.
\bigskip

We can now prove \cref{Thm:CLT_Multiset}, using the normality criterion given in \cref{Thm:WDG}.
We assume that the sequence of multiset partitions $M^{(m)}$ is regular.
This implies the estimate $e_{\ell-1}(M)=\Theta(n^{\ell-1})$, so that $T_h=\O(n^{\ell-1})$;
therefore we can take $Q_n=n^{\ell-1}$ in \cref{Thm:CLT_Multiset}.
Combining with \cref{prop:LowerBoundVarMP}, we find
$\frac{Q_n}{\sigma_n} \le \frac1{\sqrt{\Kv}} \frac{n^{\ell-1}}{n^{\ell-3/4}} =\O(n^{-1/4})$.
Furthermore, $R_n=\Theta(e_\ell(M))=\O(n^{\ell})$.
Summing up, for $s=5$,
we have
\[   \left(\tfrac{R_n}{Q_n}\right)^{1/s}\, \tfrac{Q_n}{\sigma_n} 
=\O\big(n^{1/5} n^{-1/4}\big)=\O(n^{-1/20}).
\]
This tends to $0$ as $n$ tends to infinity, and \eqref{EqHypoMainThm} is satisified.
We conclude that $\OccMP_{M^{(m)}}$ is asymptotically normal, as wanted. \qed

\subsection{Variance estimate for pattern counts in multiset permutations}
\label{sec:VarianceMP}

As above, we fix a pattern $\tau$ of size $\ell$.
Throughout this section, we assume that $M^{(m)}$ is an $\ell$-regular sequence of multisets.
We will most of the time drop to superscript $m$ and denote $n=|M^{(m)}|$.
The goal of this section is to prove \cref{prop:LowerBoundVarMP},
that is to bound $\Var(\Occ^\tau(\bm\word))$ from below,
where $\bm\word$ is a random uniform multiset permutation of $M$.
We use the notation $v(M)= \Var(\Occ^\tau(\bm\word))$,
making the dependence in $\tau$ implicit.

\subsubsection{An initial bound and a recursive inequality on the variance}
Let $M=\{1^{a_1},2^{a_2},\dots\}$ be a multiset of size $n$
and choose $j_0$ minimal such that $a_{j_0} >0$.
We denote $M'$ the multiset obtained by removing a single copy of $j_0$ from $M$.

Then a random uniform multiset permutation $\bm\word$ of $M$ can be constructed as follows.
\begin{itemize}
  \item Choose $P$ uniformly at random between $1$ and $n$ and set $\bm\word_P=j_0$.
  \item Take a uniform random multiset permutation $\bm\word'$ of $M'$, independent from $P$,
    and fill the other positions of $\bm\word$
    in the same order as in $\bm\word'$. Formally, we set, for $i \ne P$,
    \[ \bm\word_i=\begin{cases}
      \bm\word'_i &\text{ if }i <P;\\
      \bm\word'_{i-1} &\text{ if }i >P.
    \end{cases}\]
\end{itemize}
It is straightforward to check that, by construction, $\bm\word$ is a uniform random multiset permutation
of $M$. Moreover, we can write $\OccMP(\bm\word) = B + C$, where
\begin{itemize}
  \item $B$ is the number of occurrences of $\tau$ in $\bm\word$, 
    {\em not using} the position $P$ -- it is easy to see that 
    $B=\OccMP(\bm\word')$ --;
  \item $C$ is the number of occurrences of $\tau$ in $\bm\word$,
    using the position $P$.
\end{itemize}
Using the law of total variance, we have
\[v(M)=\Var(\OccMP(\bm\word)) = \esper\big[ \Var(\OccMP(\bm\word)|P) \big]
    + \Var\big[ \esper(\OccMP(\bm\word)|P) \big].\]
Note that $B$ is independent of $P$, so that $\esper(B|P)$ is the {\em constant} random variable
a.s. equal to $\esper(B)$. This implies that
\[ \Var\big[ \esper(\OccMP(\bm\word)|P) \big] 
= \Var\big[ \esper(B) + \esper(C|P) \big] = \Var\big[\esper(C|P) \big].\]
In particular, we get an initial bound
\begin{equation}
  v(M) \ge \Var\big[\esper(C|P) \big].
  \label{eq:InitialBound}
\end{equation}
On the other hand, from Cauchy-Swartz and Jensen inequalities,
since $B$ and $P$ are independent,
we have
\[ \Big| \esper\big[\Cov(B,C |P)\big] \Big|
\le \esper\big[ \sqrt{\Var(B)} \sqrt{\Var(C|P)} \big] 
\le \sqrt{\Var(B)} \, \sqrt{\esper[\Var(C|P) ]}.\]
Expanding $\Var(\OccMP(\bm\word)|P)=\Var(B+C|P)$ by bilinearity, we find that
\begin{align*}
  \esper\big[ \Var(\OccMP(\bm\word)|P) \big] &
= \Var(B) + 2\esper\big[\Cov(B,C |P)\big] + \esper\big[\Var(C|P)\big] \\
&\ge \Var(B) -2 \sqrt{\Var(B)} \sqrt{\esper[\Var(C|P)]}.
\end{align*}
Note that $\Var(B)=v(M')$ since $B=\OccMP(\bm\word')$.
Summing up, we get the following recursive inequality on $v(M)$:
\begin{equation}
  v(M) \ge v(M') \, \bigg( 1- 2\sqrt{\tfrac{ \esper[\Var(C|P)]}{v(M')} }\, \bigg)
  + \Var\big[\esper(C|P) \big]. 
  \label{eq:RecursiveInequality}
\end{equation}

\subsubsection{Analysing the initial bound \eqref{eq:InitialBound}}
Recall that $C$ counts the number of occurrences of $\tau$ in $\bm\word$ that use the position $P$.
Since the letter at position $P$ is the smallest one in $\bm\word$,
it should correspond to $1$ in the pattern $\tau$.
Formally, if $r$ is the index such that $\tau(r)=1$,
an occurrence $(i_1,\dots,i_\ell)$ of $\tau$ in $\bm\word$ using the position $P$
must satisfy $i_r=P$.
We therefore have, conditionally on $P$,
\begin{equation}
   C=\sum_{i_1 < \dots <i_{r-1}<P \atop i_\ell> \dots >i_{r+1}>P} 
   \left[ \sum_{j_\ell > \dots >j_2 >j_1=j_0}
X_{i_1}^{j_{\tau(1)}} \dots X_{i_\ell}^{j_{\tau(\ell)}} \right]. 
\label{eq:C_Condition_P}
\end{equation}
Note that the factor $X_{i_r}^{j_{\tau(r)}}=X_P^{j_0}=1$ because of the condition on the summation index.
For any $i_1,\dots,i_{r-1},i_{r+1},\dots,i_\ell$
we have
\begin{multline}
  \esper\big[X_{i_1}^{j_{\tau(1)}} \dots X_{i_{r-1}}^{j_{\tau(r-1)}}\, X_{i_{r+1}}^{j_{\tau(r+1)}} 
  \dots X_{i_\ell}^{j_{\tau(\ell)}}|P\big] \\
= \esper\big[(X')_{i_1}^{j_{\tau(1)}} \dots (X')_{i_{r-1}}^{j_{\tau(r-1)}}\, (X')_{i_{r+1}-1}^{j_{\tau(r+1)}} 
\dots (X')_{i_\ell-1}^{j_{\tau(\ell)}} \big], 
\label{eq:X_Xprime}
\end{multline}
where the variables $(X')_i^j$ refer to the multiset permutation $\bm\word'$.
In particular, the right-hand side is a quantity $F(j_2,\dots,j_\ell)$ that depends neither on $P$,
nor $i_1,\dots,i_\ell$ (the uniform random multiset permutation $\bm\word'$ is invariant by re-indexing).
Since the number of choices for the indices $i_1,\dots,i_{r-1},i_{r+1},\dots,i_\ell$ 
is $\binom{P-1}{r-1}\binom{n-P}{\ell-r}$,
we have
\[\esper(C|P) = \binom{P-1}{r-1}\binom{n-P}{\ell-r}
\left(\sum_{j_\ell > \dots >j_2 >j_1=j_0} F(j_2,\dots,j_\ell) \right).\]
We take the variance of this function of the random variable $P$
(the sum is independent of $P$, i.e. deterministic):
\begin{equation}
  \Var\big[  \esper(C|P) \big] =
\left(\sum_{j_\ell > \dots >j_2 >j_1=j_0} F(j_2,\dots,j_\ell) \right)^2 
\, \Var\left( \binom{P-1}{r-1}\binom{n-P}{\ell-r} \right).
\label{VarExpC}
\end{equation}
Since $P$ is uniformly distributed in $\{1,\dots,n\}$,
we see easily that
the variance of the product of binomials is of order $n^{2\ell-2}$.
Moreover, \cref{eq:joint_moments_Mset} gives us 
(recall that the $j_t$ are distinct here)
\[ F(j_2,\dots,j_\ell) = \frac{a_{j_2} \dots a_{j_t}}{(n-1)\dots(n-\ell+1)}
\ge {n^{-\ell+1}} \prod_{t=2}^\ell a_{j_t}.\]
Therefore the sum in \eqref{VarExpC} is bigger than $e_{\ell-1}(M\setminus{(j_0)})\, n^{-\ell+1}$,
where $M\setminus{(j_0)}$ is obtained form $M$ by removing {\em all} copies of $j_0$.
Note that, if a sequence $M^{(m)}$ of multiset partition is $\ell$ regular, 
then, $e_{\ell-1}(M\setminus{(j_0)}) = \Theta(n^{\ell-1})$.
We conclude that 
there exist $\Ka,\Kab>0$ such that
\begin{equation}
  \Var\big[  \esper(C|P) \big] \ge \Ka (e_{\ell-1}(M\setminus{(j_0)}))^2 \ge \Kab n^{2\ell-2}.
  \label{eq:Tech7}
\end{equation}
In particular, using \eqref{eq:InitialBound}, we have
\begin{equation}
  v(M) \ge \Kab n^{2\ell-2}.
  \label{eq:InitialBoundExplicit}
\end{equation}

\subsubsection{Analysing the recursive inequality}
The lower bound \eqref{eq:InitialBoundExplicit} is not sufficient to apply \cref{Thm:WDG} directly.
We shall use the recursive inequality \cref{eq:RecursiveInequality} to improve it.
To this end, we first need to analyse the term $\esper\big[\Var(C|P)\big]$.

\begin{lemma}
Let $M^{(m)}$ be an $\ell$-regular sequence of multisets.
There exist $\Kb,\Kbb>0$ such that
\[\esper\big[\Var(C|P)\big] \le \Kb\, e_{\ell-1}(M\setminus{(j_0)}) e_{\ell-2}(M\setminus{(j_0)})
\le \Kbb n^{2\ell-3}.\]
  \label{lem:Tech}
\end{lemma}
\begin{proof}
We start from \cref{eq:C_Condition_P,eq:X_Xprime}. 
Conditionally on $P$, we have (note that the indices $i_t$ for $t \ge r+1$ below are shifted by $1$,
compared to \cref{eq:C_Condition_P,eq:X_Xprime}):
\[
   C=\sum_{{i_1 < \dots <i_{r-1}<P \atop i_\ell> \dots >i_{r+1} \ge P} \atop
   j_\ell > \dots >j_2 >j_1=j_0}
 (X')_{i_1}^{j_{\tau(1)}} \dots (X')_{i_{r-1}}^{j_{\tau(r-1)}}\, (X')_{i_{r+1}}^{j_{\tau(r+1)}} 
\dots (X')_{i_\ell}^{j_{\tau(\ell)}}.
 \]
 In the sequel, we use $\bm i$ and $\bm j$ to represent lists $(i_1,\dots,i_{r-1},i_{r+1},\dots,i_\ell)$
 and $(j_2,\dots,j_\ell)$ as in the above summation index.
 We also write $(\bm X')_{\bm i}^{\bm j}$ for the corresponding monomial.
We have
\[\Var(C|P) = \sum_{\bm i,\bm{\tilde i},\bm j,\bm{\tilde j}} \Cov\left( (\bm X')_{\bm i}^{\bm j},
(\bm X')_{\bm {\tilde i}}^{\bm {\tilde j}} \right).\]
Since $\bm \word'$ is a uniform random multiset permutation of $M'$,
the variable $(\bm X')_{\bm i}^{\bm j}$ admit a $(\bm \Psi,\bD)$-weighted dependency graph
(by \cref{Thm:CLT_Multiset,prop:WDG_Products}) and we have that, for some constant $\Kb>0$
\[\Big|\Cov\left( (\bm X')_{\bm i}^{\bm j},
(\bm X')_{\bm {\tilde i}}^{\bm {\tilde j}} \right) \Big|
\le \Kb \, \bm \Psi \Big( (\bm X')_{\bm i}^{\bm j},             
(\bm X')_{\bm {\tilde i}}^{\bm {\tilde j}} \Big)
\, W\left( (\bm X')_{\bm i}^{\bm j},             
(\bm X')_{\bm {\tilde i}}^{\bm {\tilde j}} \right).\]
With the same case distinction as in \cref{ssec:proof_MSet}
on whether $(\bm X')_{\bm i}^{\bm j}$ and $(\bm X')_{\bm {\tilde i}}^{\bm {\tilde j}}$
share a lower index, an upper index or no index at all,
we can prove that
\[ |\Var(C|P)| \le \Kb \, e_{\ell-1}(M\setminus{(j_0)}) \, e_{\ell-2}(M\setminus{(j_0)}).\]
The upper bound is independent of $P$, so that
\[\esper\big[\Var(C|P)\big] \le \Kb\, e_{\ell-1}(M\setminus{(j_0)}) \, e_{\ell-2}(M\setminus{(j_0)}).\]
For $\ell$-regular sequences of multisets,
the upper bound is smaller than $\Kbb n^{2\ell-3}$ (for some $\Kbb>0$),
concluding the proof of the lemma.
\end{proof}
Plugging in \cref{eq:InitialBoundExplicit,eq:Tech7,lem:Tech} 
in \cref{eq:RecursiveInequality}, 
we get that, for some constant $\Kc>0$,
\begin{equation}
   v(M) \ge v(M') ( 1 - \Kc n^{-1/2})+\Kab n^{2\ell-2}. 
   \label{eq:RecursiveInequalityExplicit}
 \end{equation}
We denote $M_{(i)}=((M')\cdots)'$ the partition obtained by removing
the $i$ smallest parts of $M$ (counting parts with repetitions).
For $i \le \Kd n^{1/2}$ (where $\Kd$ is a positive constant that will be determined later), 
the sequence $M_{(i)}$ is still $\ell$-regular.
Therefore, we can apply \eqref{eq:RecursiveInequalityExplicit} and we have: for 
\[ v(M_{(i)}) \ge v(M_{(i+1)}) ( 1 - \Kc (n-i)^{-1/2})+\Kab (n-i)^{2\ell-2}.\]
We start from the initial inequality $v(M_{(\Kd \sqrt{n})}) \ge \Kab (n-\Kd \sqrt{n})^{2\ell-2}$
and iterate the above recursive inequality: we get
\[ v(M) \ge \sum_{i=0}^{\Kd \sqrt{n}-1} \Kab (n-i)^{2\ell-2} (1 - \Kc (n-i)^{-1/2}) \cdots (1 - \Kc n^{-1/2})\]
For $\Kd$ sufficiently small, the product is bounded away from $0$, so that each of the $\Theta(\sqrt{n})$ terms in the sum
behaves as $\Theta(n^{2\ell-2})$.
Therefore, there exists a constant $\Kv>0$ such that
\begin{equation}
   v(M) \ge \Kv n^{2\ell-3/2},
   \label{eq:FinalInequalityVaraince}
 \end{equation}
 which is exactly what we wanted to prove.

 \begin{remark}
   \label{rmk:ImprovingVarianceBound}
   Plugging in the final inequality \eqref{eq:FinalInequalityVaraince} in \eqref{eq:InitialBoundExplicit},
   we could improve \eqref{eq:RecursiveInequalityExplicit} to
   \[   v(M) \ge v(M') ( 1 - \Kc n^{-\bm{3/4}})+\Kab n^{2\ell-2}. \]
   Then, arguing as above with $i \le \Kdb n^{3/4}$, we have $v(M) \ge \Keb n^{2\ell-5/4}$
   (for some $\Keb>0$).
   Iterating this argument shows that for any $\eps>0$, 
   there exists $\Kec>0$ such that $v(M) \ge \Kec\, n^{2\ell-1-\eps}$.
   However, the bound $v(M) \ge \Kv n^{2\ell-3/2}$ given above is sufficient to apply our normality criterion.
 \end{remark}

 \subsubsection{Comparison with an upper bound}
 \label{ssec:discussion_variance_bounds}

 \cref{prop:WDG_UpperBound_Var} and the estimates $R=\O(e_\ell(M))$, $T_1=\O(e_{\ell-1}(M))$
 from \cref{ssec:proof_MSet}
yield the following upper bound on the variance,
which is valid without the regularity hypothesis.
\begin{proposition}
  \label{prop:UpperBoundVarMP}
  There exists a constant $\Ku >0$, such that for each multiset permutation
    $M=\{1^{a_1},2^{a_2},\dots\}$, we have  
    \[\Var(\OccMP) \le \Ku e_{\ell}(M) e_{\ell-1}(M). \]
\end{proposition}
We note in the case $\tau=21$, {\em i.e.}, when we are interested in inversions
in random multiset permutations, this upper bound is tight (up to a multiplicative constant):
see \cite[Eq. (1.10) or Lemma 3.1]{canfield11mahonian}.
It is natural to ask whether this is also the case for longer patterns.
\begin{question}
  \label{question}
  Fix a pattern $\tau$. Does there exist a constant $\Kl >0$ (depending on $\tau$),
  such that for each multiset permutation
      $M=\{1^{a_1},2^{a_2},\dots\}$, we have                               
      \begin{equation}
        \Var(\OccMP) \ge \Kl\, e_{\ell}(M) e_{\ell-1}(M) \ ?
        \label{eq:QuestionLowerBound}
      \end{equation}
\end{question}
An affirmative answer to this question would imply that the central limit in \cref{Thm:CLT_Multiset}
holds under the less restrictive condition that the sequence of multiset permutations
satisfy $n_{(\ell-1)} \to \infty$ ($n_{(j)}$ is defined in \eqref{eq:DefNj}).
This condition is easily seen to be necessary for having a central limit theorem
(see \cite[Section 5]{canfield11mahonian} for the case of inversions).

For a regular sequence $M=M^{(m)}$ of multisets, 
we have the estimate $e_{\ell}(M) e_{\ell-1}(M)=\Theta(n^{2\ell-1})$,
and we can prove $\Var(\OccMP) \ge \Kec\, n^{2\ell-1-\eps}$ 
for arbitrarily small $\eps>0$; see \cref{rmk:ImprovingVarianceBound}.
Therefore, in this case, the suggested lower bound \eqref{eq:QuestionLowerBound} holds,
at least up to subpolynomial factors.

\section{Arc patterns in set partitions}
\label{part:set_partitions}
In this section, we prove \cref{Thm:CLT_SetPart}.
We refer to \cref{ssec:ThmSetPart} for notation.

\subsection{Stam's urn model and the weighted dependency graph}
\label{ssec:StamWDG}
Stam's urn model \cite{stam1983RandomSetPartition}
gives a simple way to uniformly sample from $\ptn([n])$. It works as follows:
the first step consists in picking the number of urns $M$ according to the distibution 
\begin{equation}
  \P(M=m)=\frac{1}{eB_n}\frac{m^n}{m!},
  \label{eq:distribution_M}
\end{equation}
where $B_n=|\ptn([n])|$ is the $n$-th Bell number.
This is indeed a probability measure, thanks to Dobi\'nski's formula; 
see \cite{rota64partitions} for an insightful proof.

In the second step, we drop each number $i\in[n]$ into one of $M$ urns with uniform probability $1/M$.
We denote by $\pi$ the random set partition of $[n]$ in which $i$ and $j$ are in the same block
if and only if they were dropped into the same urn.
This construction of a random set partition will be referred to as Stam's urn model.

\begin{proposition}[\cite{stam1983RandomSetPartition}]
  The random set partition $\pi$ constructed by Stam's urn model is uniformly distributed on $\ptn([n])$.
  Moreover, the number of empty urns in the process has law $\mathrm{Poisson}(1)$ and
  is independent from $\pi$.
  \label{prop:Stam}
\end{proposition}

For $1 \le i<j \le n$, we consider the random variable
\[X_{ij}=\One [\text{there is an arc from $i$ to $j$ in $\pi$}].\]
\begin{theorem}
  \label{thm:WDG_SetPart}
 Consider the weighted graph $G^n$ on vertex-set $A^n$
 with weights
 \[w(X_{ij},X_{i'j'}) = \begin{cases}
   1 &\text{ if }i = i' \text{ or }j = j';\\
   1/n &\text{ otherwise.}
 \end{cases}\]
 Then, for each $n  \ge 1$, the graph $G^n$ is a 
 $(\Psi,\bm{C}_n)$ weighted dependency graph
 for the family $A^n$, where $\Psi$ is the function on multisets
 of elements of $A^n$ defined by
 \[ \Psi(B)=n^{-\#(B)}. \]
 and $(\bm{C}_n)_{n \ge 1}=(C_{r,n})_{r \ge 1,n \ge 1}$ is a doubly indexed sequence of coefficients
 such that for each $r$, we have $C_{r,n}=\tO(1)$ as $n$ tends to infinity.
\end{theorem}
The proof is postponed to \cref{sec:WDG_SetPartition}, but let us say a few words on its structure.
We should bound joint cumulants $\kappa(X_{ij}, \, X_{ij} \in B)$ for multiset $B$ of elements $A^n$;
with the same kind of observations as in the proof of \cref{thm:WDG_MSet}, we can assume w.l.o.g.
that $B$ is a (repetition-free) subset of $A^n$.
The difficulty here is that the joint moments of $X_{ij}$ do not have a simple form as in \cref{eq:joint_moments_Mset}.
However, conditionally on $M$, joint moments do have a nice multiplicative expression.
Indeed, a random set partition $\pi$ generated by Stam's urn model contains all the arcs
in some set $B$ if and only if:
\begin{enumerate}
  \item for every arc $(i,j)$ in $B$,
   we drop its endpoint $j$ in the same urn as its starting point $i$;
 \item for every arc $(i,j)$ in $B$, none of the integers between $i$ and $j$
   is dropped in the same urn as $i$;
   equivalently,
   if $g$ is not the endpoint of an arc of $B$,
   it should be dropped in a different urn from the 
   starting points of those arcs in $B$ that go over $g$.
\end{enumerate}
Since all balls are dropped uniformly independently, this happens with probability
\begin{equation}
  \mathbb E\bigg( \prod_{X_{ij} \in B} X_{ij} \big|M \bigg)
  =\frac{1}{M^a} \prod_{g: g \notin \{j,\, X_{ij} \in B\}} \frac{M-a(g)}{M},
   \label{eq:JointMoment_SP}
 \end{equation}
 where $a(g)$ is the number of arcs in $B$ going over $g$.
 Using the multiplicative form of \cref{eq:JointMoment_SP}
 and general results from the theory of weighted dependency graphs, we can show that
 the {\em conditional} joint cumulants $\kappa\big( X_{ij} \in B|M\big)$ are small.
 
 To go back to unconditional cumulants, we use 
 Brillinger's \defn{law of total cumulance}
 \cite{brillinger1969totalcumulance}, which we now recall.
If $X_1,\dots,X_r$ and $Y$ are random variables with finite moments
defined on the same probability space, then
\begin{equation}
   \kappa(X_1,\dots,X_r)
= \sum_\rho \kappa\Big( \kappa(X_i, i \in B | Y), B \in \rho\Big),
\label{eq:TotalCumulance}
\end{equation}
where the sum runs over all set partitions $\rho$ of $[r]$
and the $B$'s are the blocks of $\rho$.

Note that the inner conditional cumulants are functions of the random variable $Y$, 
that is, in our setting, of $M$. We therefore need estimates
for cumulants of particular functions of $M$, which we derive
in \cref{ssec:CumM} as a preliminary to the proof of \cref{thm:WDG_SetPart}.
\medskip

In the rest of this section, we admit \cref{thm:WDG_SetPart} and show how
to deduce \cref{Thm:CLT_SetPart} from it.

\subsection{A lower bound for the variance}
\label{ssec:Var_SetPart}
As in the statement of \cref{Thm:CLT_SetPart},
we fix an arc pattern $\mathcal A$ of length $\ell$ 
with $\arcs$ arcs and denote $\OccSP$ its number of occurrences
in a uniform random set partition of size $n$.
We would like a lower bound on its variance.

To this end, we use the law of total variance, conditioning on the value of 
the number $M$ of urns in Stam's construction:
\[\Var(\OccSP)= \esper\big[ \Var(\OccSP|M) \big] + \Var\big[ \esper(\OccSP|M) \big].\]
Both terms are nonnegative; it turns out that the second one is relatively easy to evaluate
and gives us a good enough lower bound (in fact, this lower bound is optimal
up to logarithmic factors, as explained at the end of \cref{ssec:CLT_SP}).
\begin{lemma}
  $\displaystyle \Var(\OccSP) \ge \Var\big[ \esper(\OccSP|M) \big] 
  = \tTheta\Big( n^{2\ell-2\arcs-1} \Big).$
  \label{lem:LowerBound_Var_SetPart}
\end{lemma}

\begin{proof}
We start by introducing notation.
For $i$ in $\{2,\dots,\ell\}$, we denote by $a_i$ the number of 
arcs above the segment $[i-1,i]$ in $\mathcal A$. 
For example if $\mathcal A$ is the arc pattern of \cref{fig:ExArcRepr},
then $a_2=a_5=1$ and $a_3=a_4=2$.
By convention, we set $a_1=0$. 

Recall from \cref{eq:JointMoment_SP} above that
\begin{equation}
  \mathbb P\big( (x_1,\dots,x_\ell) \text{ occurrence of }\mathcal A |M \big)
  =\frac{1}{M^a} \prod_{g: g \notin \{x_j, (i,j) \in \mathcal A\}} \frac{M-a(g)}{M}.  
   \label{eq:Cond_Prob_Occ}
 \end{equation}
The product can be split in two parts.
\begin{itemize}
  \item First, consider factors 
    indexed by $g$ in $\{x_1,\dots,x_\ell\} \setminus \{x_j, (i,j) \in \mathcal A\}$,
    i.e. by elements of the pattern that are not the ending point of an arch.
    There are a fixed number of such $g$, and the corresponding
    numbers $a(g)$ do not depend on $x_1$, \dots, $x_\ell$ nor on $n$.
    Thus the product of such factors
    is a Laurent polynomial of degree $0$ in $M$, say $R_0(M)$,
    independent of $x_1,\dots,x_j$ (and of $n$);
  \item Second, we consider factors corresponding to $g \notin \{x_1,\cdots,x_\ell\}$:
    for each $i$ in $[\ell]$,
    we get $x_{i}-x_{i-1}-1$ factors equal to $\frac{M-a_i}{M}$ (by convention $x_0=0$).
\end{itemize}
Hence \eqref{eq:Cond_Prob_Occ} rewrites as
\[
  \mathbb P\Big( (x_1,\dots,x_\ell) \text{ occurrence of }\mathcal A \big|M \Big)=
\frac{R_0(M)}{M^a} \prod_{i \le \ell} \left( \frac{M-a_i}{M} \right)^{x_i-x_{i-1}-1}.\]
Summing this conditional probability over $x_1<x_2<\cdots<x_\ell \le n$, we get the conditional expectation
of $\OccSP$. For convenience, we rather write this formula with summation indices $y_i=x_i-x_{i-1}$
(with the convention $x_0=0$, {\em i.e.} $y_1=x_1$):
\[
  \esper\big[ \OccSP|M\big] = \frac{R_0(M)}{M^a} \sum_{y_1,\ldots,y_\ell \ge 1 \atop y_1+\ldots+y_\ell \le n}
  \prod_{i \le \ell} \left( \frac{M-a_i}{M} \right)^{y_i-1}.
 \]
We denote $t$ the number of $i$ such that $a_i \ne 0$.
Since the above expression is symmetric in the $a_i$, we may assume as well that $a_1,\cdots,a_t$ are nonzero,
while $a_{t+1}=\cdots=a_\ell=0$. The summand in the above display does not depend on $(y_i, i>t)$,
so that we can write
\begin{equation}
  \esper\big[ \OccSP|M\big] = \frac{R_0(M)}{M^a}
  \sum_{y_1,\ldots,y_t \ge 1 \atop y_1+\ldots+y_t \le n}
  \binom{n-y_1-\cdots-y_t}{\ell-t} \prod_{i \le t} \left( 1-\frac{a_i}{M} \right)^{y_i-1}.
\label{eq:CondEsper_BigSum}
\end{equation}
The next step is to see how the sum behaves when $n \to \infty$,
and $M$ is closed to its expectation \[m_n:=\mathbb E(M) \sim n/\ln(n).\]
We will see later that $M$ concentrates around $m_n$,
see beginning of \cref{ssec:CumM}.
Informally the dominant term is obtained as follows:
\begin{enumerate}
  \item replace the binomial coefficient
by its dominant part, which is
\[\tfrac{1}{(\ell-t)!}\, (n-y_1-\cdots-y_t)^{\ell-t};\]
\item forget the condition $y_1+\ldots+y_t \le n$
in the sum;
\item use the approximation
  \begin{multline*}
    \sum_{y \ge 1} y^j \left( 1-\frac{a}{M} \right)^{y-1} \approx
  \sum_{y \ge 1} (y+j-1) \cdots y \left( 1-\frac{a}{M} \right)^{y-1}
  \\
  =
  j! \left(\frac{1}{1-\left( 1-\frac{a}{M} \right)}\right)^{j+1}
  = j! \left(\frac{M}{a}\right)^{j+1}.
\end{multline*}
\end{enumerate}
(For the middle equality, note that both sides are the $j$-th derivative of $1/(1-x)$
evaluated in $x=1-\frac{a}{M}$.)
With this heuristic, we get
\begin{multline*}
  \sum_{y_1,\ldots,y_t \ge 1 \atop y_1+\ldots+y_t \le n}                    
  \binom{n-y_1-\cdots-y_t}{\ell-t} \prod_{i \le t} \left( 1-\frac{a_i}{M} \right)^{y_i-1}\\
  \approx \frac{1}{(\ell-t)!} \sum_{j_0,\dots,j_t \ge 0 \atop j_0+\dots+j_t=\ell-t}
  \binom{\ell-t}{j_0,\cdots,j_t}\, n^{j_0} \prod_{i \le t} \left[ \sum_{y_i \ge 1} (-y_i)^{j_i} \left( 1-\frac{a_i}{M} \right)^{y_i-1} \right]
  \\
    \approx \sum_{j_0,\dots,j_t \ge 0 \atop j_0+\dots+j_t=\ell-t} (-1)^{j_1+\dots+j_t}
  \frac{n^{j_0}M^{j_1+\dots+j_t+t}}{j_0! a_1^{j_1+1} \cdots a_t^{j_t+1}}. 
\end{multline*}
  The summand corresponding to $j_0=\ell-t,\, j_1=\dots=j_t=0$
  dominates the above sum.
  Inserting this estimate back into \cref{eq:CondEsper_BigSum},
  we deduce heuristically the following:
  uniformly on $M$ in $[m_n-n^{3/4},m_n+n^{3/4}]$, one has
\begin{equation}\label{eq:Equiv_Cond_Exp}
  \esper\big[ \OccSP|M\big] =  \frac{n^{\ell-t} M^{t-a}}{(\ell-t)!} (1+o(1)).
\end{equation}
Taking the variance of this function of $M$, we expect to get (see \cref{ssec:CumM} for a justification of this heuristic):
\begin{equation}
   \Var\big(\esper\big[ \OccSP|M\big] \big)
   = \esper\bigg[ \esper\big[ \OccSP|M\big]^2 \bigg] \tTheta(n^{-1}) = \tTheta\Big( n^{2\ell-2\arcs-1} \Big). 
  \label{eq:VarToJustify} 
\end{equation}
  In \cref{sec:variance_estimate}, we prove formally
  \cref{eq:Equiv_Cond_Exp,eq:VarToJustify}.
In particular, this completes the proof of the lemma.
\end{proof}

\subsection{The central limit theorem}
\label{ssec:CLT_SP}
In this section, we prove \cref{Thm:CLT_SetPart}.
The quantity $\OccSP$ of interest can be written in terms of the $X_{ij}$ 
as follows
\[ \OccSP = \sum_{s_1 < s_2 < \dots <s_\ell} 
\prod_{(i,j)\in \mathcal A} X_{s_i s_j}. \]
From \cref{thm:WDG_SetPart,prop:WDG_Products},
the family of monomials $\prod_{(i,j)\in \mathcal A} X_{s_i s_j}$
has a $(\bm \Psi_n,\bm D_n)$-weighted dependency graph, where
\begin{itemize}
  \item $\bm \Psi_n$ is the function on multisets of monomials
    $\prod_{(i,j)\in \mathcal A} X_{s_i s_j}$, which is $n^{-\# \text{arcs}}$,
    where $\# \text{arcs}$ denotes the total number of distinct arcs 
    $X_{s_i s_j}$ appearing in some monomial in the multiset;
  \item $(\bm D_n)_{n \ge 1}=(D_{r,n})_{r \ge 1,n \ge 1}$ is 
    a doubly indexed sequence of numbers
    such that, for each $r \ge 1$, we have $D_{r,n}=\tO(1)$.
\end{itemize}
We will prove the central limit theorem for $\OccSP$
by applying \cref{Thm:WDG}. For this, we need to evaluate the quantities $R_n$
and $T_{h,n}$ defined in \cref{EqDefR,EqDefT}.
We have
\[R_n:=\sum_{1\le s_1<\dots<s_\ell \le n}
\Psi\left( \left\{ \prod_{(i,j)\in \mathcal A} X_{s_i s_j} \right\}\right)
=\Theta(n^{\ell-a}).\]
(Recall that $a$ is the number of arcs in $\mathcal A$.)
We now consider $T_{h,n}$. In the following, 
we use $\alpha^{(k)}$ or $\beta$ to denote some $\ell$-uple of positive integers $s_1<\dots<s_\ell \le n$
and $\Pi(\alpha^{(k)})$ or $\Pi(\beta)$
for the corresponding monomial $\prod_{(i,j)\in \mathcal A} X_{s_i s_j}$.
By definition,
\[ T_{h,n}= \max_{\alpha^{(1)},\dots,\alpha^{(h)}} \left[ \sum_{\beta} F(\beta) \right],\]
where 
\[F(\beta)=W(\{\Pi(\beta)\},\{\Pi(\alpha^{(1)}),\dots,\Pi(\alpha^{(h)})\}) 
\frac{\Psi\big(\{\Pi(\beta),\Pi(\alpha^{(1)}),\dots,\Pi(\alpha^{(h)})\}\big)}
{\Psi\big(\{\Pi(\alpha^{(1)}),\dots,\Pi(\alpha^{(h)}) \}\big)}.  \]
To evaluate this quantity, we fix $\alpha^{(1)},\dots,\alpha^{(h)}$.
All constants in $\O$ terms below depend implicitly on $h$,
but not on $\alpha^{(1)},\dots,\alpha^{(h)}$.
\begin{itemize}
  \item Consider first the terms where $\beta$
    is disjoint from $\alpha^{(1)} \cup \dots \cup \alpha^{(h)}$.
    For such $\beta$, the weight in $F(\beta)$ is $1/n$,
    while the quotient of $\Psi$ is $n^{-a}$.
    Thus, we have $F(\beta)=n^{-1-a}$.
    Since there are $\O(n^\ell)$ such terms,
    their total contribution is $\O(n^{\ell-a-1})$.
  \item Consider now terms where $\beta$ 
    intersects $\alpha^{(1)} \cup \dots \cup \alpha^{(h)}$,
    but there is no factor in common between $\Pi(\beta)$
    and any $\Pi(\alpha^{(k)})$.
    There are only $\O(n^{\ell-1})$ such terms.
    In this case we bound the weight in $F(\beta)$ by $1$,
    and the quotient of $\Psi$'s is still $n^{-a}$.
    The total contribution of such terms is therefore also
    $\O(n^{\ell-a-1})$.
  \item We finally consider terms,  where some factors
    of $\Pi(\beta)$, say $g$ of them, are already present in some
    $\Pi(\alpha^{(k)})$ (possibly different factors are in different  $\Pi(\alpha^{(k)})$).
    This forces $\beta \cap (\alpha^{(1)} \cup \dots \cup \alpha^{(h)})$
    to be of size at least $g+1$,
    so that the number of such $\beta$ is $\O(n^{\ell-g-1})$.
    Again we bound the weight in $F(\beta)$ by $1$,
    but now the quotient of $\Psi$'s is $n^{-a+g}$.
    The total contribution of such terms is therefore also   
        $\O(n^{\ell-a-1})$ as well.
\end{itemize}
From this discussion, for any fixed $h\ge 1$,
we have $T_{h,n}=\O(n^{\ell-a-1})$.
We can therefore set $Q_n=n^{\ell-a-1}$ in \cref{Thm:WDG}.

From \cref{lem:LowerBound_Var_SetPart}, we know
that 
\[\sigma_n:=\sqrt{\Var(\OccSP)} \ge \tTheta(n^{\ell-a-1/2}).\]
Therefore, for $s=3$, we have
\[  \left(\tfrac{R_n}{Q_n}\right)^{1/s}\, \tfrac{Q_n}{\sigma_n} =\tTheta(n^{-1/6}) \]
and thus, this quantity tends to $0$ faster than any power of $D_{r,n}$ (for any fixed $r \ge 1$).
From \cref{Thm:WDG} and \cref{Rmk:WDGNonCstC}, we conclude that $\OccSP$ is asymptotically normal.

We still need to justify the expectation and variance estimates.
From \cref{eq:Equiv_Cond_Exp} above and using that $\mathbb E(M^{t-a})=\tTheta(n^{t-a})$
(see \cref{eq:momentM} below), we have $\mathbb E(\OccSP)= \tTheta(n^{\ell-a})$.
For the variance, a lower bound is obtained in \cref{lem:LowerBound_Var_SetPart}.
A matching upper bound comes from \cref{prop:WDG_UpperBound_Var},
using the estimates $R_n=\Theta(n^{\ell-a})$ and $T_{1,n}=\O(n^{\ell-a-1})$ given above.
\qed

\section{The weighted dependency graphs for set partitions}
\label{sec:WDG_SetPartition}
The goal of this section is to prove \cref{thm:WDG_SetPart},
i.e. that a given weighted graph is a weighted dependency graphs
for the presence of arcs in set partitions.


\subsection{Two general simple estimate for cumulants}
We start by two easy bounds on cumulants.
Denote $B_r=\sum_{\pi\in\ptn([r])} |\mu(\pi,{[r]})|$, 
which is a universal constant depending only on $r$.
We use the standard notation for the $r$-norm $\|X\|_r:=\E[|X|^r]^{1/r}$.
\begin{lemma}
  For any random variables $X_1,\dots,X_r$ with finite moments defined on the same probability space,
  we have
  \begin{equation}
     |\kappa(X_1,\dots,X_r)| \le B_r \prod_{i=1}^r \|X_i\|_r.
     \label{eq:Bound_Cumu_Holder}
   \end{equation}
   \label{lem:Bound_cum_Holder}
\end{lemma}
\begin{proof}
  This is an immediate consequence of the moment-cumulant formula \eqref{eq:def_cumulants},
  combined with H\"older inequality and the monotonicity of $r$-norms
  ($\|X\|_s \le \|X\|_r$ is $s \le r$).
\end{proof}
For an event $A$, we denote 
$\compl{A}$ its complement.
\begin{lemma}
  Let $X_1,\dots,X_r$ be random variables on the same probability space.
  Then, for any event $A$, we have
  \begin{align}
    |\esper(\One[A] \,X_1\, \cdots \,X_r )-\esper (X_1\, \cdots \,X_r )| &\le 
    \P(\compl{A})^{1/(r+1)} \, \prod_{i=1}^r \|X_i\|_{r+1};\\
    |\kappa(X_1 \One[A],\dots,X_r \One[A])-\kappa(X_1,\dots,X_r)| 
    &\le r \, B_r\, \P(\compl{A})^{1/(r+1)} \prod_{i=1}^r \|X_i\|_{r+1} \, .
    \label{eq:Bound_Cumu_Indicator}
  \end{align}
  \label{lem:cum_Moment_Rare}
\end{lemma}
\begin{proof}
  The first formula is a trivial consequence of H\"older inequality:
  \begin{multline*}
  \esper(\One[A] \,X_1\, \cdots \,X_r )-\esper (X_1\, \cdots \,X_r )=
  \esper(\One[\compl{A}] \, X_1\, \cdots \,X_r ) \\
  \le \|\One[\compl{A}]\|_{r+1} \, \prod_{i=1}^r \|X_i\|_{r+1}.
  \end{multline*}
  For the second, we use the moment-cumulant formula \eqref{eq:def_cumulants} and write
  \begin{multline}
    \label{eq:Tech4}
  \kappa(X_1 \One[A],\dots,X_r \One[A])-\kappa(X_1,\dots,X_r) \\
  = \sum_{\pi\in\ptn([r])}\mu(\pi,{[r]}) \prod_{B \in \pi} \left[ \esper\left(\One[A] \prod_{i \in B} X_i \right)
  - \esper\left( \prod_{i \in B} X_i \right)\right].
\end{multline}
For a block $B$ of size $s$ we have, since $r$-norms are increasing in $r$,
\[
  \bigg|\esper\left(\One[A] \prod_{i \in B} X_i \right) \bigg| 
  \le \bigg| \esper\left( \prod_{i \in B} |X_i| \right) \bigg| \le \prod_{i \in B} \|X_i\|_s \le \prod_{i \in B} \|X_i\|_{r+1};
  \]
  and
  \begin{multline*}
  \left| \left[ \esper\left(\One[A] \prod_{i \in B} X_i \right)     
    - \esper\left( \prod_{i \in B} X_i \right)\right] \right| 
    \le \P(\compl{A})^{1/(s+1)} \prod_{i=1}^r \|X_i\|_{s+1} \\
    \le  \P(\compl{A})^{1/(r+1)} \prod_{i \in B} \|X_i\|_{r+1}.
\end{multline*}
 Combining this with the classical inequality
  \[|a_1 \dots a_t - b_1 \dots b_t| 
  \le \sum_{i=1}^t |b_1| \cdots |b_{i-1}| \, |a_i - b_i| \, |a_{i+1}| \dots |a_t|\]
  implies that, for any set partition $\pi$, we have
  \[\left| \prod_{B \in \pi} \left[ \esper\left(\One[A] \prod_{i \in B} X_i \right)
    - \esper\left( \prod_{i \in B} X_i \right)\right] \right|
    \le |\pi| \P(\compl{A})^{1/(r+1)}\prod_{i=1}^r \|X_i\|_{r+1}.\]
    Since we always have $|\pi| \le r$ , plugging this inequality back into \cref{eq:Tech4}
    proves \cref{eq:Bound_Cumu_Indicator}.
\end{proof}

\subsection{Cumulants of rational functions in $M$}
\label{ssec:CumM}


In this section, we bound (joint) cumulants of various functions of $M$.
The important recurrent feature is that all cumulants 
have a smaller order of magnitude than what we could naively expect,
the difference being a factor $n^{1-r}$ 
(for joint cumulants of order $r$, up to logarithmic factors).
As usual, constants in $\O$ and $\tO$ symbols do depend
on the order $r$ of the cumulant under consideration.

Recall that the distribution of $M$ is given in \eqref{eq:distribution_M}.
In particular $M$ depends on $n$, even if this is implicit in the notation.
For a fixed integer $r$ (either positive or negative), one has
\begin{equation}
  \esper[M^r]=\frac{B_{n+r}}{B_n}=\tTheta(n^r);
  \label{eq:momentM}
\end{equation}
the first equality is indeed a direct consequence of the formula
for the distribution of $M$ (\cref{eq:distribution_M}),
while the second equality follows from asymptotic results for Bell numbers,
see {\em e.g.} \cite[Eq. (4)]{canfield1995engel}.
This implies $\|M\|_r=\O(n)$.
Regarding cumulants, the following estimates for $r \le 3$ were given in \cite[Section 8.3]{ModGaussian1}:
\begin{align}
  \label{eq:asymp_mn}  m_n:=\esper[M]&=\frac{n}{\ln\, n}(1+o(1))\\
  \label{eq:asymp_sin} \si_n^2:=\Var(M)&=\frac{n}{(\ln\, n)^2}(1+o(1)),\\
 \ka_3(M)&=\frac{2n}{(\ln n)^2}(1+o(1)).
\end{align}
Having such asymptotic equivalent for all cumulants seem hard,
but we can easily get a $\O(n)$ bound, which will be sufficient for us.

\begin{lemma}
 We have 
 $|\ka_r(M)| = \O(n) = m_n^r\, \tO(n^{1-r})$.
  \label{lem:Bound_CumM}
\end{lemma}
\begin{proof}
  Note that the second equality follows from \eqref{eq:asymp_mn}. We focus therefore on the first one.

  From \cref{prop:Stam}, it follows that $M$ the same law as $X_n+P$,
where $X_n$ is the number of parts in a uniform random set partition of $[n]$,
and is independent from the $\mathrm{Poisson}(1)$ random variable $P$.
Besides, we know from Harper \cite[p. 413]{HarperCLTBlocks} 
that $X_n$ can be written as a sum of $n$ independent Bernoulli variables
  $X_n=\sum_{i=1}^n B_{i,n}$; the parameters of these Bernoulli variables
  are here irrelevant.
  Summing up, we get for $M$ the following useful representation:
  \begin{equation}
    M \stackrel{\text{\tiny law}}= P+ \sum_{i=1}^n B_{i,n}.
    \label{eq:M_Sum_Ind}
  \end{equation}
 
Using the additivity of cumulants on independent random variables, we have
\[\ka_r(M) = \ka_r(P) + \sum_{i=1}^n \ka_r(B_{i,n}).\]
Since $\ka_r(B_{i,n})$ is bounded by a constant $D_r$, independently on the parameter of $B_{i,n}$,
the lemma is proved.
\end{proof}
This lemma has the following easy consequence. 
Consider the following normalized version of $M$
\[Z_n:=\frac{M-m_n}{\sigma_n},\]
then its cumulants behave as follows: $\ka_1(Z_n)=0$ and, for $r\ge 2$,
\begin{equation}
    |\ka_r(Z_n)| =\si_n^{-r} \ka_r(M)= \tO( n^{1-r/2} ).
   \label{eq:bound_Cumu_Zn}
 \end{equation}
For $r \ge 3$, the upper bound tends to $0$,
so that $Z_n$ converges in distribution and in moments
to a standard normal variable.
(The convergence in distribution is stated in \cite[Theorem 2.1]{CLT_SetPartitionsStatistics},
see also Harper \cite{HarperCLTBlocks}.)
\medskip

We now give bounds for joint cumulants of powers of $M$.
\begin{corollary}
  For any integers $i_1,\dots,i_r \ge 1$,
  we have
  \[\ka(M^{i_1},\cdots,M^{i_r}) = \O(n^{i_1+\dots+i_r-r+1}) =m_n^{i_1+\dots+i_r} \tO(n^{1-r}),\]
  where the constants in $\O$ and $\tO$ symbols depend on $i_1,\dots,i_r$.
  \label{corol:bound_Cumu_Powers}
Consequently, we have
\begin{equation}
  \ka(Z_n^{i_1},\dots,Z_n^{i_r}) = \left( \frac{m_n}{\si_n} \right)^{i_1+\dots+i_r} \tO(n^{-r+1}).
   \label{eq:Tech1}
 \end{equation}
\end{corollary}
\begin{proof}
  Leonov-Shiryaev formula for cumulants of products of random variables \cite{LeonovShiryaevCumulants} gives
\[\ka(M^{i_1},\cdots,M^{i_r}) = \sum_\pi \prod_{B \in \pi} \ka_{|B|}(M),\]
where the sum runs over all set partitions $\pi$ of $[i_1]\uplus \cdots \uplus [i_r]$ such that
\begin{equation}
  \pi \vee \{[i_1],\cdots,[i_r] \} = \{[i_1] \uplus \dots \uplus [i_r] \}.
  \label{Eq_Cond_LS}
\end{equation}
(Here, $\vee$ is the joint operation on the set partition lattice, ordered by refinement.)
Using \cref{lem:Bound_CumM}, we have
\[\ka(M^{i_1},\cdots,M^{i_r})  = \O\left( \max_\pi n^{\#(\pi)} \right),\]
where $\#(\pi)$ is the number of parts of $\pi$
and the maximum is taken on set partitions $\pi$ satisfying \eqref{Eq_Cond_LS}.
A simple combinatorial argument (left to the reader)
shows that \eqref{Eq_Cond_LS} implies 
\hbox{$\#(\pi) \le i_1+\dots+i_r+1-r$},
proving the corollary.
\end{proof}
\medskip

Finally, we consider joint cumulants of $1/M$ and $Z_n/M$,
which will be useful in the next section.
To this end, we need some concentration inequality for $M$.
Let us introduce the following event
\[A_n=\{ m_n - n^{3/4} \le M \le m_n +n^{3/4} \}.\]
Note that $n^{3/4}$ is larger than the standard deviation $\sigma_n=\tO(n^{1/2})$
of $M$, so that we expect $A_n$ to hold with large probability.
Indeed, the following holds.
\begin{lemma}
 $\P[ \compl{A_n} ]$ tends to $0$ faster than any rational function of $n$.
  \label{lem:Concentration_M}
\end{lemma}
\begin{proof}
  We use the representation \eqref{eq:M_Sum_Ind} of $M$ as a sum of independent variables:
\[    M \stackrel{\text{\tiny law}}= P+ \sum_{i=1}^n B_{i,n}, \]
where $P$ follows a Poisson law of parameter $1$
and the $B_{i,n}$ are independent Bernoulli variables,
whose parameters are not relevant.
A standard tail estimate for Poisson distribution (see, {\em e.g.}, \cite[p 97]{ProbAndComputing}) gives
\[ \P\Big[ \big|P - \E(P) \big| \ge \tfrac12 n^{3/4} \Big] 
\le  \P\Big[ P \ge \tfrac12 n^{3/4} \Big] \le e^{-1} \left( \frac{2e}{n^{3/4}} \right)^{n^{3/4}}, \]
while Hoeffding's inequality \cite[Theorem 1, eq. (2.3)]{HoeffdingInequality} tells us that
\[ \P \bigg[ \,\bigg|\sum_{i=1}^n (B_{i,n} - \E(B_{i,n}))\bigg| \ge \tfrac12 n^{3/4} \bigg] \le 2 \exp \big(-\tfrac12 n^{1/2}\big).\]
Combining both inequalities tells us yields, for $n$ big enough,
\[ \P\big[ |M-m_n| \ge n^{3/4}\big] \le 3 \exp \big(-\tfrac12 n^{1/2}\big).\qedhere\]
\end{proof}
\smallskip

 From \cref{lem:cum_Moment_Rare},
 adding/removing $\One[A_n]$ from joint cumulants/joint moments of
 powers of $Z_n$, $M$ and $1/M$ change the resulting value
 by an error that is smaller than any rational function of $n$
 (indeed, for each fixed $r \ge 1$, the $r$-norms of
 $Z_n$, $M$ and $1/M$ are bounded by $\tO(\sqrt{n})$, $\O(n)$ and $1$ respectively).
 We will denote by $oe(n)$ such error terms and use this fact repeatedly below.

 We can now turn back to bounds on cumulants.
\begin{proposition}
For nonnegative integers $r_1,r_2$ with $r_1+r_2 \ge 1$, 
we have
\[\ka\Big(\underbrace{M^{-1},\dots,M^{-1}}_{r_1 \, \text{\scriptsize \em times}},
\underbrace{Z_n/M,\dots,Z_n/M}_{r_2 \,\text{\scriptsize \em times}}\Big) =m_n^{-r_1} \sigma_n^{-r_2} \tO(n^{1-r_1-r_2}).\]
\label{prop:Cumu_MInverse}
\end{proposition}
\begin{proof}
  Fix some integer $k \ge 0$ (that will be specified later).
Elementary analysis asserts that, 
 for any $x \ge -1/2$,
 \[(1+x)^{-1}= \sum_{i=0}^{k-1} (-x)^i + \eps_k(x),\]
 where $|\eps_k(x)| \le 2 |x|^{k}$.
 When $A_n$ holds, for $n$ large enough, we can use this expansion for $x=\frac{M-m_n}{m_n}=\frac{\si_n}{m_n} Z_n$ and write
 \begin{multline*}
 \One[A_n] \, \frac{m_n}{M} =\One[A_n]\,  
\left( 1+ \frac{\si_n}{m_n} Z_n \right)^{-1} \\
= \One[A_n]\,\left[ \sum_{i=0}^{k-1} \left( - \frac{\si_n}{m_n} Z_n \right)^i
+ \eps_k \left( \frac{\si_n}{m_n} Z_n \right) \right].
\end{multline*}
Multiplying this by $\frac{\si_n}{m_n}Z_n$ and 
setting $\eps'_k(x)=x \eps_{k-1}(x)$ (so that we have the bound $|\eps'_k(x)| \le 2 |x|^{k}$, as for $\eps_k(x)$),
we get (after a shift of index)
\[\One[A_n] \, \si_n \frac{Z_n}{M} = \One[A_n]\,\left[ - \sum_{i=1}^{k-1} \left( - \frac{\si_n}{m_n} Z_n \right)^{i}
 + \eps'_k \left( \frac{\si_n}{m_n} Z_n \right) \right].
 \]
We now consider the cumulants
\begin{equation}
  k_{r_1,r_2}:=\kappa\Bigg(\underbrace{\One[A_n]\, \tfrac{m_n}{M},\dots,\One[A_n]\, \tfrac{m_n}{M}}_{r_1 \,\text{\scriptsize \em times}},
\underbrace{\One[A_n] \, \si_n \tfrac{Z_n}{M},\dots, \One[A_n] \, \si_n \tfrac{Z_n}{M}}_{r_2\,\text{\scriptsize \em times}} \Bigg) .
\label{eq:Cumulants_Strange_Normalization}
\end{equation}
One the one hand, since $\compl{A_n}$ has exponentially small probability,
we can forget the indicators $\One[A_n]$ in the definition of $k_{r_1,r_2}$,
up to an error term $oe(n)$.
Therefore, we have
\begin{equation}
  k_{r_1,r_2}=m_n^{r_1} \sigma_n^{r_2} \, \ka\Big(\underbrace{M^{-1},\dots,M^{-1}}_{r_1 \, \text{\scriptsize \em times}},
  \label{eq:Tech10}
  \underbrace{Z_n/M,\dots,Z_n/M}_{r_2 \,\text{\scriptsize \em times}}\Big) \, +\, oe(n)
\end{equation}
On the other hand, we can expand \eqref{eq:Cumulants_Strange_Normalization} by multilinearity
and bound separately each summand (the number of summands is independent of $n$).
We set $r=r_1+r_2$ and distinguish two types of summands.
\begin{itemize}
  \item First consider, for $0 \le i_1,\dots,i_r \le k-1$, the summand 
    \[ \pm\, \kappa\left[\One[A_n] \left( - \frac{\si_n}{m_n} Z_n \right)^{i_1},
    \dots, \One[A_n] \left(-\frac{\si_n}{m_n} Z_n \right)^{i_r}\right],\]
    which does not involve any of the functions $\eps_k$.
    Removing the indicators $\One[A_n]$ yields an error of order $oe(n)$.
    Then, using \eqref{eq:Tech1}, we see that such terms are of order $\tO(n^{1-r})$.
  \item We now consider summands that involves some function $\eps_k$ or $\eps'_k$.
    Again the indicator functions can be forgotten up to an error of order $oe(n)$.
    For such terms, we use \cref{lem:Bound_cum_Holder}
    and therefore we only have to bound expressions of the form
    \begin{multline*}
      \E \left[ \left|\frac{\si_n}{m_n} Z_n \right|^{r\, i_1}\right]^{1/r} \cdots \E \left[\left|-\frac{\si_n}{m_n} Z_n \right|^{r\, i_t} \right]^{1/r} \cdots \\
      \cdots
    \E \left[\left|\eps_k\left( \frac{\si_n}{m_n} Z_n \right)\right|^{r}\right]^{\tfrac{s_1}{r}} 
    \E \left[\left|\eps'_k\left( \frac{\si_n}{m_n} Z_n \right)\right|^{r}\right]^{\tfrac{s_2}{r}}
  \end{multline*}
    for triples $(t,s_1,s_2) \ne (r,0,0)$ of sum $r$
    and integers $i_1,\dots, i_t \ge 0$.
    Denoting $A=k(s_1+s_2)+i_1+\dots+i_t$ and using the bound for $\eps_k$ and $\eps'_k$, 
    the above expression is smaller than
    \[2^{s_1+s_2} \left(\frac{\si_n}{m_n}\right)^A \E[|Z_n|^{r i_1}]^{1/r} \dots \E[|Z_n|^{r i_t}]^{1/r} \E[|Z_n|^{kr}]^{\tfrac{s_1+s_2}{r}}.\]
        Since cumulants of $Z_n$ converges, so does its moments and its absolute moments,
    and we get a bound in $\tO(n^{-A/2})$. Take $k=2(r-1)$. Since $s_1+s_2 \ge 1$,
    we have $A \ge k$ and
    we get that all summands are $\tO(n^{1-r})$.
\end{itemize}
Finally we conclude that $k_{r_1,r_2}=\tO(n^{1-r})$.
Together with \cref{eq:Tech10}, this concludes the proof.
\end{proof}
\begin{corollary}
  \label{corol:bound_Cumu_Powers_MInv_Zn}
  For any nonnegative integers $i_1,\dots,i_r$ and $j_1,\dots,j_r$, we have
  \[  \ka\Bigg[ \left( \tfrac{1}{M} \right)^{i_1} \left( \tfrac{Z_n}{M} \right)^{j_1},
  \dots,\left( \tfrac{1}{M} \right)^{i_r} \left( \tfrac{Z_n}{M} \right)^{j_r} \Bigg]
  =m_n^{-i_1-\dots-i_{r}} \sigma_n^{-j_1-\dots-j_r} \tO (n^{1-r} ),\]
  where the constant in the $\tO$ symbols depend on $i_1,\dots,i_{r_1},j_1,\dots,j_{r_2}$.
\end{corollary}
\begin{proof}
  Similar to the proof of \cref{corol:bound_Cumu_Powers}.
\end{proof}
We consider polynomials $P$ in the variables $(\bm x=\{x_1,\dots,x_p\},y,z)$
and define the following gradations.
If $P$ is a {\em monomial} in these variables, we denote by $\deg_{\bm x}(P)$ its total degree in $x_1,\dots,x_p$;
by $\deg_{y}(P)$ its degree in $y$;
and by $\deg_{\bm x,y^{-1}}(P)$ the difference $\deg_{\bm x}(P)-\deg_{y}(P)$.
As usual, for any of the three above notions of degree,
the degree of a polynomial $P$ is the maximal degree of a monomial with a nonzero coefficient in $P$.

We claim that \cref{corol:bound_Cumu_Powers_MInv_Zn} implies by linearity the following bound.
Let $(P_i)_{1 \le i \le s}$ be polynomials in $(\bm x=\{x_1,\dots,x_p\},y,z)$ 
and $d_i$ be integers with $\deg_{\bm x,y^{-1}}(P_i) \le d_i$.
Then, uniformly for all  
values $\bm a=(a_1,\dots,a_p)$ and $\bm b=(b_1,\dots,b_s)$ in $[0,n]$, we have
\begin{equation}
  \kappa\Big( P_i\big(\bm a,M^{-1},\tfrac{b_i \sigma_n Z_n}{M \cdot m_n}\big)\, :1 \le i \le s \Big)
=n^{\sum_{i \le s} d_i} \tO(n^{1-s}). 
\label{eq:BoundCumuPoly}
\end{equation}
Indeed it is enough to check this for monomials $P_i=\bm x^{\bm e_i}\, y^{f_i}\, z^{g_i}$
(for $i$ in $[s]$),
the general case following by linearity.
For monomials we have, using \cref{corol:bound_Cumu_Powers_MInv_Zn}
\begin{align*}
  \Bigg| \kappa\Big( P_i\big(\bm a,M^{-1},\tfrac{b_i \sigma_n Z_n}{M \cdot m_n}\big)
  \, :1 \le i \le s \Big) \Bigg|
  \hspace{-4cm} & \hspace{4cm} \\
  & \le \prod_{i \le s} \big( \bm a^{\bm e_i} (\tfrac{b_i \sigma_n}{m_n})^{g_i} \big)
   \cdot \Big| \kappa\big( M^{-f_i} (Z_n/M)^{g_i} \, : i \le s \big) \Big|\\
   & \le n^{\sum_{i \le s} \deg_{\bm x}(P_i)} \, \left( \tfrac1{\sqrt{n}} \right)^{g_1 +\dots+g_s}
   m_n^{-f_1-\dots-f_s} \sigma_n^{-g_1-\dots-g_s} \tO (n^{1-s}) \\
   & \le n^{\sum_{i \le s} (\deg_{\bm x}(P_i)-f_i)} \tO(n^{1-s}),
 \end{align*}
 concluding the proof of \eqref{eq:BoundCumuPoly}.

\subsection{Proof of \cref{thm:WDG_SetPart}}
  By definition of weighted dependency graphs,
we have to prove that 
for any multiset \hbox{$B=\{X_{i_1 j_1},\ldots,X_{i_r j_r}\}$} of elements of $A^n$,
one has
\begin{equation}
  \bigg| \ka\big( X_{i_1 j_1},\ldots,X_{i_r j_r} \big) \bigg| \le
  C_{r,n} \, \Psi(B) \, \MWST{\WDep^n[B]}. 
    \label{EqFundamentalRepeated_SetPart}
\end{equation}
Recall that $\MWST{\WDep^n[B]}$ is by definition
the maximal weight of a spanning tree of the graph
induced by $\WDep^n$ on $B$.

As in the proof of \cref{thm:WDG_MSet},
\cite[Proposition 5.2]{feray18dependency} allows us 
to reduce the proof of \eqref{EqFundamentalRepeated_SetPart}
to the case where $\WDep^n[B]$ has no edges of weight $1$,
i.e. where
$i_1,\dots,i_r$ and $j_1,\dots,j_r$ are two lists
of distinct integers (but we may have $i_t=j_s$ for some $s,t$ in $[r]$).
Indeed, this reduction was based on the following facts,
which hold true here as well:
\begin{itemize}
  \item 
variables linked by an edge of weight $1$ are incompatible 
(i.e. their product is a.s. $0$);
  \item 
our random variables take values in $\{0,1\}$;
\item $\Psi(B)$ is insensitive to repetitions.
\end{itemize}
In the case where $i_1,\dots,i_r$ and $j_1,\dots,j_r$ are two lists
of distinct integers, the bound to be proven is
\begin{equation}
  \bigg| \ka\big( X_{i_1 j_1},\ldots,X_{i_r j_r} \big) \bigg| 
  \le C_{r,n} n^{1-2r} =\tO(n^{1-2r}),
    \label{EqToProve}
\end{equation}
where the constant in $\tO$ symbol depend on $r$ but neither on $n$
nor on the indices $i_1,\dots,i_r,j_1,\dots,j_r$.
\medskip

Fix $r\geq 2$.
Take $i_1,i_2,\ldots,i_r$ distinct 
and $j_1,j_2,\ldots,j_r$ distinct with $i_s<j_s$ for all $s$.
We identify $B$ to the set of arcs $\{(i_1,j_1),(i_2,j_2),\ldots,(i_r,j_r)\}$
and will use letters $C$ and $D$ for subsets of $B$.

We first condition on the number of urns $M$ and get bound for the conditional cumulants.
Consider a set partition $\pi$ generated through the urn model with $M$ urns.
As explained in \eqref{eq:JointMoment_SP} above, we have
\begin{equation}
 \E\left(\prod_{(i,j)\in{D}}X_{ij} \Big| M\right)
=\prod_{g\in[n]}p_g({D}), \label{eq:CondExp_ProdArcs}
\end{equation}
where $p_g({D})$ is the probability of $g$ being dropped in a correct urn, that is
\[p_g({D})=\begin{cases}
  \frac{1}{M},\text{ if $g=j_t$ for some $(i_t,j_t)$ in D;}\\
               \frac{M-a_g({D})}{M},\text{ otherwise}.
              \end{cases}
\]
Here $a_g({D}):=\{(i,j)\in{D}:i<g<j\}$ is the number of arcs in $D$ above $g$.

As in the proof of \cref{thm:WDG_MSet},
we will use the small-cumulant/quasi-factorization equivalence.
For a subset $C$ of $B$, we define
\[Q_{C}(M):=\prod_{{D}\subseteq{C}}\E\bigg(\prod_{(i,j)\in {D}}X_{ij} \Big| M\bigg)^{(-1)^{|{D}|}}.\]
For each ${D}\subseteq{C}$, we can use \eqref{eq:CondExp_ProdArcs} for
the corresponding joint moment and we get
\[Q_{C}(M)
=\prod_{g\in[n]}\prod_{{D}\subseteq{C}}p_g({D})^{(-1)^{|{D}|}}.\]
The integers $g$ indexing the product can be divided in three categories:
\begin{enumerate}
  \item There exists an arc $(i_0,j_0)$ in ${C}$ that is not above $g$; 
    more precisely such that $g \le i_0$ or $g>j_0$, the large and strict inequalities being important.
    Then 
    \[ \prod_{{D}\subseteq{C}}p_g({D})^{(-1)^{|{D}|}}=1,\]
    since the map ${D}\leftrightarrow {D}\bigtriangleup\{(i,j)\}$ 
    ($\bigtriangleup$ is the symmetric difference)
    gives a fixed-point free sign-reversing involution of the factors.
  \item If all arcs $(i,j)$ in ${C}$ are above $g$ 
    (i.e. $i < g < j$ for all $(i,j)$ in ${C}$),
    then for all ${D} \subseteq {C}$, we have $a_g({D})=|{D}|$.
    Therefore,
    \[ \prod_{{D}\subseteq{C}}p_g({D})^{(-1)^{|{D}|}}
    = \prod_{{D}\subseteq{C}} \left(  \frac{M-|{D}|}{M} \right)^{(-1)^{|D|}}
    = \prod_{{D}\subseteq{C}} (M-|{D}|)^{(-1)^{|D|}}. \]
    The right hand-side clearly depends only on $M$ and of the size $|C|$ of ${C}$.
    By \cite[Lemma A.1]{feray18dependency},
    it writes as $1+R_{|C|}(M)$ for some rational function
    $R_{|C|}$ of degree at most $-|C|$.
  \item The remaining case is when both $\min(j_1,\dots,j_s)>\max(i_1,\dots,i_s)$ 
    and $g=\min(j_1,\dots,j_s)$ hold. To simplify notation, assume $g=j_1$.
    Then $p_g({D})=\frac1M$ whenever $(i_1,j_1) \in {D}$. The corresponding factors cancel out
    and we get, with similar arguments as above for the other factors,
    \[ \prod_{{D}\subseteq{C}}p_g({D})^{(-1)^{|{D}|}} =
    \prod_{{D}\subseteq{C} \setminus \{(i_1,j_1)\}} (M-|{D}|)^{(-1)^{|D|}}
    = 1+R_{|C|-1}(M). \]
\end{enumerate}
In conclusion, we have
\[Q_{C}(M)=(1+R_{|C|}(M))^{\ell_{C}}\, (1+R_{|C|-1}(M))^{\One[\ell_{C}>0]},\]
where $\ell_{C}$ is the number of $g\in[n]$ that are below all arcs, {\em i.e.}
\[\ell_{C}:=\max\Big[ \big(\min(j_1,\dots,j_s)-\max(i_1,\dots,i_s)\big),\, 0\Big].\]
Note that we always have $\ell_{C} \leq n$.
\medskip

We now perform an asymptotic expansion of $Q_C$,
using $\ast$ for coefficients which depend only on $|C|$ and do not need to be made explicit:
if $\ell_C>0$, we have
\begin{multline*}(1+R_{|C|}(M))^{\ell_{C}} =
\exp\big[ \ell_C \log(1+R_{|C|}(M)) \big] \\
= \exp\big[ \ell_C \log(1 + \ast M^{-|C|} + \ast M^{-|C|+1} + \dots + \O(M^{-r})) \big]\\
= \exp\big[ \ast \ell_C M^{-|C|} + \ast \ell_C M^{-|C|+1} + \dots + \O(\ell_C M^{-r})) \big] \\
=1 + \ast \ell_C M^{-|C|}+ \dots + \O(\ell_C M^{-r})+ \ast \ell_C^2 M^{-2|C|} + \dots + \O( \ell_C^{r} M^{-r|C|}).
\end{multline*}
For $|C| \ge 2$, under the assumption that $M \ge \tfrac12 m_n$,
both error terms are smaller than $\tO(n^{1-r})$
(we use the universal bound $\ell_C \le n$).
The main term is a polynomial in $\ell_C$ and $M^{-1}$.
After substracting $1$, its total degree in $\ell_C$ and $M$ is at most $-|C|+1$.
Note that both the constants in the error terms
and this polynomial only depend on the size of $|C|$.
The same applies to the other factor $1+R_{|C|-1}(M)$ of $Q_C$.
We therefore conclude that, for all $C$ such that $\ell_C>0$,
\begin{equation}
  Q_{C}(M) = 1 + P_{|C|}(\ell_C,M^{-1}) + \tO_{|C|}(n^{1-r}),
  \label{eq:Expansion_Quotients}
\end{equation}
where $P_{|C|}$ is a polynomial in $x$ and $y$
such that $\deg_{\bm x,y^{-1}}(P_{|C|}) \le -|C|+1$
and the error term $\tO_{|C|}(n^{1-r})$ is uniform for all $C$ of a given size
and all $M \ge \tfrac12 m_n$.
For $\ell_C=0$, we have $Q_{C}(M) = 1$, so that \eqref{eq:Expansion_Quotients}
holds trivially in this case as well (though, with different polynomials $P_{|C|}$
and error terms $\tO_{|C|}(n^{1-r})$).
\smallskip

We shall also need an expansion of the conditional expectation of single variable $X_{ij}$:
\begin{multline} \esper\big[ X_{ij} \, |\, M \big]
=  M^{-1} \big(1 - M^{-1} \big)^{j-i-1} \\
= M^{-1} \exp \big[ -(j-i)M^{-1} \big] \exp\big[M^{-1}+ \ast(j-i-1)M^{-2} + \dots + \O( (j-i) M^{-r}) \big]
\label{eq:Tech2}
 \end{multline}
 The second exponential writes as $\tilde{P}_1(j-i,M^{-1})+\tO(n^{1-r})$ 
 for some polynomial $\tilde{P}_1$ with $\deg_{\bm x,y^{-1}}(\tilde P_1) \le 0$.
 The exponent of the first exponential however is typically large,
 forcing us to use the concentration property of $M$ and
 exhibit the deterministic dominant term. We have
 \[\exp \big[ -(j-i)M^{-1} \big]
 = \exp \big[ -(j-i) m_n^{-1} \big] \exp \Bigg[ \frac{(j-i) \sigma_n Z_n}{M \cdot m_n} \Bigg], \]
 where we recall that $\sigma_n Z_n= M-m_n$ by definition.
 Assume now that we have $|M-m_n| \le n^{3/4}$,
 which implies $Z_n=\tO(n^{1/4})$.
 Then the argument of the second exponential is $\tO(n^{-1/4})$
 and we can perform a series expansion, uniformly on $(i,j)$ and $M$:
 \begin{equation}
   \exp \big[ -(j-i)M^{-1} \big]        
   = \exp \big[ -(j-i) m_n^{-1} \big] \, \Big( \hat{P}_1\big( \tfrac{(j-i) \sigma_n Z_n}{M \cdot m_n} \big)
 +\tO(n^{1-r}) \Big),
 \label{eq:Tech3}
 \end{equation}
 for some polynomial $\hat{P}_1(z)$. 
 Combining \cref{eq:Tech2,eq:Tech3}, we get
 \begin{equation}
   \esper\big[ X_{ij} \, |\, M \big]
   =  M^{-1}  \exp \big[ -(j-i) m_n^{-1} \big] 
   \Big[ P_1\big(j-i,M^{-1}, \tfrac{(j-i) \sigma_n Z_n}{M \cdot m_n} \big)
    + \tO(n^{1-r}) \Big],
   \label{eq:Expansion_expectations}
 \end{equation}
 where $P_1$ is a polynomial with $\deg_{\bm x,y^{-1}}(P_1) \le 0$
 and the error term is uniform on $(i,j)$ and $M$ in the interval 
 $[m_n-n^{3/4};m_n+n^{3/4}]$.
\medskip

We now come back to the conditional cumulant
\[\kappa_C(M):=\kappa\big(X_{ij}:(i,j)\in C \, |\, M\big).\]
As observed in \cite[Proof of Proposition 5.8]{feray18dependency},
it can be written in terms of conditional expectations
and quotients $Q$ as follows:
\[ \kappa_C(M) = \Bigg( \prod_{(i,j) \in C} \esper\big[ X_{ij} \, |\, M \big] \Bigg)\cdot
\sum_{\{{C}_1,\ldots,{C}_m\}} (Q_{{C}_1}-1) \cdots (Q_{{C}_m}-1),\]
where the sum is over all sets of (distinct) subsets 
${C}_1,\ldots,{C}_m\subseteq{C}$ such that $|{C}_i|\geq 2$ 
and the hypergraph with edges ${C}_1,\dots,C_m$ is connected.
Using the asymptotic expansions \eqref{eq:Expansion_Quotients}
and \eqref{eq:Expansion_expectations},
we get 
\begin{equation}
  \!
  \kappa_C(M) = M^{-|C|} \prod_{(i,j) \in C}\exp \big[ -(j-i) m_n^{-1} \big]
\Big[ P_C\big( \ll_C,M^{-1}, \tfrac{(j-i) \sigma_n Z_n}{M \cdot m_n} \big)
    + \tO(n^{1-r}) \Big],
    \label{eq:Expansion_Cumulants}
  \end{equation}
where $\ll_C$ represent the vector $(\ell_D)_{D \subseteq C}$
and $P_C$ is a polynomial in the variables $( \bm{x}=(x_D)_{D \subseteq C}, y,z)$.
Furthermore, the error term $\tO(n^{1-r})$ is uniform for $M$ 
in the interval $[m_n-n^{3/4};m_n+n^{3/4}]$.
Both the polynomials $P_C$ and the constants in the error terms
depend only on $C$ through its size and which of $\ell_D$ ($D \subseteq C$)
are non-zero.
Therefore the error term can thus be chosen uniformly on $C$ with $|C| \le r$.
Moreover, the total degree of $P_C$ in $\bm x$ and $y^{-1}$ is at most
\[ \max\limits_{\{ C_1,\ldots,C_m\}} (-|C_1|+1) + \dots + (-|C_m|+1), \]
where the maximum is taken over sets of subsets of $C$
such that the hypergraph with edges ${C}_1,\dots,C_m$ is connected.
It is a simple combinatorial exercise to see that this condition implies
\[(-|C_1|+1) + \dots + (-|C_m|+1) \le -|C|+1, \]
so that $\deg_{\bm x,y^{-1}}(P_C) \le -|C|+1$.
\medskip

We now use the law of total cumulance \eqref{eq:TotalCumulance}.
Recall that, from \cref{lem:Concentration_M},
the event $A_n=\{|M-m_n| \le n^{3/4}\}$
has probability $1-oe(n)$.
We have
\begin{align*} \kappa\left(X_{ij}:(i,j)\in B\right)
 &=\sum_{\pi\in\ptn(B)}\kappa\big(\kappa_C(M) : C\in\pi\big)\\
 &= \sum_{\pi\in\ptn(B)}\kappa \big(\One[A_n] \, \kappa_C(M) : C\in\pi\big) +oe(n)
\end{align*}
When $A_n$ is satisfied, the conditional cumulants $\kappa_C(M)$
can be evaluated by \eqref{eq:Expansion_Cumulants}
and we get
\begin{multline*}
   \kappa\left(X_{ij}:(i,j)\in B\right) +oe(n)
 =  \left( \prod_{(i,j) \in B} \exp \big[ -(j-i) m_n^{-1} \big] \right) \cdot \\
  \cdot \sum_{\pi\in\ptn(B)}\kappa \Bigg(\One[A_n] M^{-|C|} 
  \Big[ P_C\big( \ll_C,M^{-1}, \tfrac{(j-i) \sigma_n Z_n}{M \cdot m_n} \big)
     + \tO(n^{1-r}) \Big] \, :C \in \pi \Bigg)
\end{multline*}
We now bound the right-hand side.
The product of exponential factors is simply bounded by $1$.
Each summand of the sum over set partitions is bounded as follows.
We expand by multilinearity the joint cumulant.
\begin{itemize}
  \item One of the terms is 
    \[\kappa \Bigg(\One[A_n] \, M^{-|C|} \,                        
     \Big[ P_C\big( \ll_C,M^{-1}, \tfrac{(j-i) \sigma_n Z_n}{M \cdot m_n} \big)\Big]: C \in \pi\Bigg) 
     \]
     Using \cref{eq:BoundCumuPoly} and the fact $\deg_{\bm x,y^{-1}}(P_C) \le -|C|+1$
     we get that this term is $\tO(n^e)$, where
     \[e \le 1-|\pi| + \sum_{C \in \pi} (-2|C|+1) =1-2r.\]
   \item The other terms are joint cumulants of random variables,
     at least one of which is $\tO(n^{1-r})$ uniformly on the event $A_n$.
     Since all $P_C$ have nonpositive degree in $\bm x$ and $y^{-1}$,
     they are bounded on the event $A_n$ by $\tO(1)$.
     Moreover, on $A_n$, we have $M^{-|C|}=\tO(n^{-|C|})$.
     We therefore use the easy bounds \eqref{eq:Bound_Cumu_Holder} 
     and \eqref{eq:Bound_Cumu_Indicator} for cumulants
     and get that all these terms are 
     \[\tO(n^{1-r-\sum_{C \in \pi} |C|})=\tO(n^{1-2r}).\]
\end{itemize}
To conclude, we get that 
\[ \kappa\left(X_{ij}:(i,j)\in B\right) =\tO(n^{1-2r}).\]
Note that the constants in the $\tO$ symbols only depend on which $\ell_D$ are nonzero,
for subsets $D$ of $B$, and hence can be chosen uniformly for all sets $B$ of size $r$
(assuming, as we have always done so far,
that $i_1,\dots,i_r$, resp $j_1,\dots,j_r$, are distinct).
We have thus proved \eqref{EqFundamentalRepeated_SetPart} for 
lists $i_1,\dots,i_r$ and $j_1,\dots,j_r$ with distinct entries.
As argued at the beginning of the proof,
this ends the proof of \cref{thm:WDG_SetPart}.
\qed

\section{Technical statements for the variance estimate}
\label{sec:variance_estimate}

The goal of this section is to prove the estimates on
$\esper[\OccSP|M]$ and its variance given in \cref{eq:Equiv_Cond_Exp,eq:VarToJustify}.
According to \eqref{eq:CondEsper_BigSum}, we need to understand sums of the following kind
(where we take $a_1$, \dots, $a_t$ to be positive integers and $\ell \ge t$):
\[F_{a_1,\dots,a_t;\ell}(n,M)=
  \sum_{y_1,\ldots,y_t \ge 1 \atop y_1+\ldots+y_t \le n}
  \binom{n-y_1-\cdots-y_t}{\ell-t} \prod_{i \le t} \left( 1-\frac{a_i}{M} \right)^{y_i-1}.
\]
Note that they satisfy the following recursive formula (setting $y=y_t$)
\begin{equation}
  F_{a_1,\dots,a_t;\ell}(n,M)
= \sum_{1\le y \le n} F_{a_1,\dots,a_{t-1};\ell-1}(n-y,M) 
\, \left( 1-\frac{a_t}M \right)^{y-1}
\label{eq:Recursion_F}
\end{equation}
The general strategy is the following:
using this recursive formula, we will prove that such functions
belong to some specific graded space $V$, analyze their highest degree terms,
and show that the expectation and  the variance of an element of $V$ can be bounded above given its degree.
\medskip

We let $V$ to be the $\mathbb Q(M)[n]$-span of $(1-b/M)^n$, where $b$ runs over the set of nonnegative integers
(it is easy to see that the functions $(1-b/M)^n$ are linearly independent over $\mathbb Q[n](M)$).

We now define a gradation $\deg$ on $V$. 
Both variables $n$ and $M$ are of degree $1$.
Moreover, the degree of a term $(1-b/M)^n$ is chosen to be $-b$.
With this convention we can check
that an element of $V$ of degree at most $d$ behaves as $\tO(n^d)$,
as $n$ and $M$ to infinity with $|M-m_n| \le n^{3/4}$. 

\begin{lemma}
  \label{lem:FormCondExp}
  For all positive integers $a_1$, \dots, $a_t$ and $\ell \ge t$,
  the quantity $F_{a_1,\dots,a_t;\ell}(n,M)$ is in $V$.
  Moreover, 
  \[ F_{a_1,\dots,a_t;\ell}(n,M) = \sum_{j_0,\dots,j_t \ge 0 \atop j_0+\dots+j_t=\ell-t} 
  \frac{(-1)^{j_1+\dots+j_t}}{j_0! \, a_1^{j_1+1} \cdots a_t^{j_t+1}}
  n^{j_0}(M)^{j_1+\dots+j_t+t} + \Err(n,M), \] 
  where the error term $\Err(n,M)$ has degree at most $\ell-1$.
\end{lemma}
\begin{proof}
  For the sake of this proof, we introduce a second gradation $\deg'$ on $V$,
  for which $n$ and $M$ still have degree $1$ but any $(1-b/M)^n$ has degree $0$.
  Obviously for any $f$ in $V$, we have $\deg(f) \le \deg'(f)$.
  To distinguish both gradations, we will refer below to $\deg$ as the {\em standard degree},
  while $\deg'$ is called the {\em modified degree}.

  We introduce some additional notation: let
  \[I_a[G(n,M)]= \sum_{1\le y \le n} G(n-y,M) 
  \, \left( 1-\frac{a}M \right)^{y-1}\]
  be the $\mathbb Q(M)$-linear operator appearing in \cref{eq:Recursion_F}.
   Finally $\NoRn$ is the $\mathbb Q(M)[n]$-linear projection from $V$ to $V$ which sends the basis element $1$ to $1$
   and $(1-b/M)^n$ to $0$ for $b>0$.
   \medskip

   The proof of the lemma involves three claims.
\smallskip

   {\em Claim 1.} 
   The functional $I_a$ maps $V$ to $V$ and increases the modified degree by at most $1$.
\smallskip

   {\em Proof of claim 1.}
   By $\mathbb Q(M)$-linearity, it suffices to consider $I_a[G(n,M)]$, where $G(n,M)=P(n) \, (1-b/M)^n$
   has modified degree $d$ (i.e. $\deg(P)d$).
  By definition and elementary manipulations,
  \begin{equation}
    I_a\big[P(n) \, (1-b/M)^n\big]=(1-b/M)^{n-1}\sum_{1 \le y \le n} P(n-y)  \left( \frac{1-a/M}{1-b/M} \right)^{y-1}.
    \label{eq:Tech6}
 \end{equation}
 If $a=b$, then the above sum is a polynomial of degree $\deg(P)+1=d+1$ in $n$
 and therefore $I_a\big[P(n) \, (1-b/M)^n\big]$ is an element of $P$ of modified degree $d+1$.

 Consequently, we focus on the case $a \ne b$.
  The quantity $P(n-y)$ is a polynomial in $n$ and $y$ 
  and can be decomposed as 
  \begin{equation}
    P(n-y)=\sum_{j \le \deg(P)} \alpha_j \tilde P_j(n) y (y+1)\dots (y+j-1),
    \label{eq:Tech11}
 \end{equation}
  for some constants $\alpha_j$ and polynomials $\tilde P_j$ satisfying $\deg(\tilde P_j)+j \le \deg(P)$.
  This implies 
  \[I_a\big[P(n) \, (1-b/M)^n\big]= (1-b/M)^{n-1} \sum_{j \le \deg(P)} \alpha_j \tilde P_j(n) S_j,\]
  where
  \begin{align}
    S_j & :=  \sum_{1 \le y \le n} y (y+1)\dots (y+j-1) \left( \frac{1-a/M}{1-b/M} \right)^{y-1} 
    = \frac{d^j}{dx^j} \frac{1-x^{n+j}}{1-x} \bigg|_{x=\frac{1-a/M}{1-b/M}}.
    \label{eq:Tech5}\\
    & = \bigg( \frac{j!}{(1-x)^{j+1}}  - \sum_{h=0}^j \binom{j}{h} \, h!\, \frac{(n+j)!}{(n+h)!} \,
    \frac{x^{n+h}}{(1-x)^{h+1}}\bigg)\bigg|_{x=\frac{1-a/M}{1-b/M}}
    \nonumber 
  \end{align}
  We set $R(M):=\frac{1-a/M}{1-b/M}$.
  The above formula shows that each $S_j$ is a $\mathbb Q(M)[n]$ linear combination of $1$ and $R(M)^n$,
  so that $(1-b/M)^{n-1} S_j$ is a $\mathbb Q(M)[n]$ linear combination of $(1-b/M)^n$
  and $(1-b/M)^n R(M)^n=(1-a/M)^n$. This proves that $I_a\big[P(n) \, (1-b/M)^n\big]$ is in $V$.

Let us analyse its modified degree: $R(M)$ has degree $0$, while $\frac{1}{1-R(M)}$ has degree $1$ in $M$.
Since both $(1-b/M)^n$ and $(1-b/M)^n R(M)^n=(1-a/M)^n$ have modified degree $0$,
we see that the modified degree of $(1-b/M)^n S_j$ is at most $j+1$.
From the inequality $\deg(\tilde P_j)+j \le \deg(P)$,
we conclude that $I_a\big[P(n) \, (1-b/M)^n\big]$ has modified degree at most $\deg(P)+1=d+1$. 
  This ends the proof of Claim 1.\qed
\medskip

{\em Claim 2.} For $a>0$, we have $\NoRn \circ I_{a} = \NoRn \circ I_{a}\circ \NoRn$.
\smallskip

{\em Proof of claim 2.} Again, by $\mathbb Q(M)$-linearity,
it is again to check the equality applied to $G(n,M)=P(n) \, (1-b/M)^n$.
If $b=0$, this is trivial since, in this case, $\NoRn(G(n,M))=G(n,M)$.
For $b >0$, $\NoRn(G(n,M))=0$ and we should check that $\NoRn \circ I_{a}(G(n,M))=0$.

This is a consequence of \cref{eq:Tech5} and the discussion after it;
indeed, for $b>0$, the quantity $I_{a}\big[P(n) \, (1-b/M)^n\big]$
is a $\mathbb Q(M)[n]$ linear combination of $(1-b/M)^n$ and $(1-a/M)^n$.
Since $a$ and $b$ are both positive, these two basis elements are sent to $0$
by $\NoRn$ proving Claim 2. \qed
\medskip

{\em Claim 3.} Let $G(n,M)$ be a homogeneous element of degree $d$ in \bm{$\mathbb Q(M)[n]$}
(i.e. without terms of the form $(1-b/M)^n$).
Then, the top homogeneous component of $\NoRn \circ I_a[G(n,M)]$ is $\Phi_a[G(n,M)]$,
where $\Phi_a$ is the $\mathbb Q(M)$ linear map defined by
\[\Phi_a\left( \frac{n^d}{d!}\right)
= \sum_{j=0}^d (-1)^j \frac{n^{d-j}}{(d-j)!}  \frac{M^{j+1}}{a^{j+1}}.\]
\smallskip

{\em Proof of claim 3.}
We focus on the case $G(n,M)=P(n)=n^d$, the genral case following by $\mathbb Q(M)$-linearity.
In this case the decomposition \eqref{eq:Tech11} writes
\begin{multline*}
  P(n-y)= (n-y)^d = \sum_{j=0}^d \binom{d}{j} (-1)^j n^{d-j} y^j\\ = \sum_{j=0}^d \binom{d}{j} (-1)^j n^{d-j} y(y+1)\dots(y+j-1) 
+ \text{smaller degree terms}.
\end{multline*}
Consequently,
\[I_a\big[n^d\big] = \sum_{j=0}^d \binom{d}{j} (-1)^j n^{d-j} S_j + \text{smaller modified degree terms},\]
where $S_j$ is given by \cref{eq:Tech5}, setting $b=0$.
From \cref{eq:Tech5}, we observe that
\[\NoRn(S_j)=\frac{j!}{(1-x)^{j+1}}\bigg|_{x=1-a/M}=j! \frac{M^{j+1}}{a^{j+1}}.\]
Therefore, we have
\[\NoRn \circ I_a\big[n^d\big] = \sum_{j=0}^d \binom{d}{j} (-1)^j n^{d-j} j! \frac{M^{j+1}}{a^{j+1}},\]
concluding the proof of Claim 3. \qed
\medskip

We come back to the proof of the lemma.
Unwrapping the recursion in \cref{eq:Recursion_F}, we have
\[F_{a_1,\dots,a_t;\ell}(n,M) = I_{a_1} \circ I_{a_2} \cdots \circ I_{a_t} \left[ \binom{n}{\ell-t} \right]\]
From Claim 1, it immediately follows that 
$F_{a_1,\dots,a_t;\ell}(n,M)$ is in $V$ and has modified degree at most $\ell$.

We are interested in its homogeneous component of standard degree $\ell$.
Since we know that its modified degree is at most $\ell$,
this top homogeneous part cannot contain any term $(1-b/M)^n$ with $b>0$;
indeed such term have a bigger modified degree as standard degree.
We can therefore consider 
\begin{align*}
  \NoRn(F_{a_1,\dots,a_t;\ell}(n,M))&=\NoRn \circ I_{a_1} \circ I_{a_2} \cdots \circ I_{a_t} \left[ \binom{n}{\ell-t} \right]\\
&= \NoRn \circ I_{a_1} \circ \NoRn \circ I_{a_2} \cdots \NoRn \circ I_{a_t} \left[ \binom{n}{\ell-t} \right],
\end{align*}
where the second equality follows from Claim 2.

Since $\binom{n}{\ell-t}$ has degree $\ell-t$ and that each one of the $t$ operators
$\NoRn \circ I_{a_s}$ increases the degree by at most $1$,
we can only keep the highest degree component all along the computation.
Therefore, using Claim 3, the highest degree component of $F_{a_1,\dots,a_t;\ell}(n,M)$ is
\[ \Phi_{a_1} \circ \Phi_{a_2} \cdots \circ \Phi_{a_t} \left[ \frac{n^{\ell-t}}{(\ell-t)!} \right]
= \sum_{j_0,\dots,j_t \ge 0 \atop j_0+\dots+j_t=\ell-t} 
\frac{(-1)^{j_1+\dots+j_t}}{j_0!\, a_1^{j_1+1} \cdots a_t^{j_t+1}}
n^{j_0}M^{j_1+\dots+j_t+t}.\]
This ends the proof of the lemma.
\end{proof}

In particular the above lemma implies that,
$\Err(n,M) = \tO(n^{\ell-1})$ as $N$ tends to infinity,
uniformly on all $M$ satisfying $|M-m_n| \le n^{3/4}$.
This proves \eqref{eq:Equiv_Cond_Exp}.
\bigskip

We are now interested in controlling the variance of $\esper[\OccSP|M]$.
We start by a general result for the variance of some element in $V$.
\begin{lemma}
  \label{lem:Varf}
  Let $f(n,M)$ be a function in $V$ of degree $d$.
  Assume that $f(n,M)$ is bounded by a polynomial function of $n$ and $M$.
  Then \[\Var\big( f(n,M) \big) = \tO(n^{2d-1}).\]
\end{lemma}
\begin{proof}
  Note that any polynomial function of $n$ and $M$ is bounded in $2$-norm
  by a polynomial of $n$
  (recall that $\|M\|_r=\O(n)$ for any $r \ge 1$).
  Moreover, from \cref{lem:Concentration_M},
  the event $A_n=\{|M-m_n| \le n^{3/4}\}$ has probability $1-oe(n)$.
  Therefore 
  \[\Var\big(f(n,M)\big)=\Var\big(f(n,M) \One_{A_n}\big)+oe(n).\]
  and we aim to prove that 
  \[\Var\big(f(n,M) \One_{A_n}\big)=\tO(n^{2d-1}).\]

  By $\mathbb Q[n]$-bilinearity and Cauchy-Schwartz inequality, 
  it is enough to consider the case where $f(n,M)=R(M) (1-b/M)^n$, 
  for some $b \in \mathbb Z_{\ge 0}$.
  Performing a Taylor series expansion of the logarithm,
  we write, for some integer $A>0$,
  \[
     (1-b/M)^n=\exp\big[n\log(1-b/M) \big] 
= \exp\big[\tfrac{-b\, n}{M} + n P_1(M^{-1}) +\O(M^{-A}) \big],
\]
where $P_1$ is a polynomial in $M^{-1}$ with no constant and linear terms.
Continuing the computation, we have
\[(1-b/M)^n = \exp\big[\tfrac{-b \, n}{m_n}\big]\, \exp\big[\tfrac{-b \, n(M-m_n)}{M m_n}\big]\, 
\exp\big[n P_1(M^{-1}) \big]\, \big(1+\O(M^{-A}) \big). \]
When $|M-m_n| \le n^{3/4}$, we have $\tfrac{-b \, n(M-m_n)}{M m_n}=\tO(n^{-1/4})$
and $n P_1(M^{-1}) =\tO(n^{-1})$.
Doing some series expansion of the second and third exponential factors above, we get
\[\frac{(1-b/M)^n}{\exp\big[\tfrac{-b \, n}{m_n}\big]} =P(n,M,m_n) +\tO(n^{-A}),\]
where $P(n,M,m_n)$ is a Laurent polynomial of total degree $0$.
We now multiply by $R(M)$ and perform a series expansion of $R(M)$ on the right hand side:
this gives
\[\frac{R(M)\, (1-b/M)^n}{\exp\big[\tfrac{-b \, n}{m_n}\big]} =\tilde{P}(n,M,m_n) +\tO(n^{d_R-A}),\]
where $\tilde P(n,M,m_n)$ is a Laurent polynomial of total degree $d_R:=\deg(R)$.
Since, for all integers $r$, we have
$\Var(M^r)=n^{2r} \tO(n^{-1})$ (see \cref{corol:bound_Cumu_Powers,corol:bound_Cumu_Powers_MInv_Zn}),
it is clear that all Laurent monomials degree $d$ in $n$, $M$ and $m_n$
have variance $\tO(n^{2d-1})$.
Hence using Cauchy-Schwartz inequality and choosing $A=1$,
we have
\begin{align*}
  \Var\left( \frac{R(M)\,(1-b/M)^n}{\exp\big[\tfrac{-bn}{m_n}\big]} \One_{A_n} \right)
  &=\Var\bigg(\tilde P(n,M,m_n) +\tO\big(n^{d_R-A}\big)\bigg)\\
  &=\tO\big(n^{2d_R-1}\big).
\end{align*}
Since the denominator is deterministic and has order $n^{-b}$, this implies
\[\Var\Big(R(M)\, (1-b/M)^n\, \One_{A_n} \Big) = \tO\big(n^{2(d_R-b)-1}\big),\]
as wanted.
\end{proof}
\begin{corollary}
  \label{corol:Var}
  If $R_0(M)$ denotes the Laurent polynomial appearing in \eqref{eq:CondEsper_BigSum},
  we have 
  \[\Var\left( \tfrac{R_0(M)}{M^a} \Err(n,M))\right)=\tO(n^{2\ell-2a-3}).\]
\end{corollary}
\begin{proof}
  We know that $\Err(n,M)$ is an element of $V$ of degree at most $\ell-1$,
  which implies that $\tfrac{R_0(M)}{M^a} \Err(n,M)$ is also an element of $V$,
  of degree $\ell-a-1$.
  Moreover, since  $ \esper(\OccSP|M)$ is bounded by a polynomial function
  in $n$ and $M$ (namely $n^{|\mathcal A|}$),
   the quantity $\tfrac{R_0(M)}{M^a} \Err(n,M)$ also has this property.
  We can therefore apply \cref{lem:Varf} and conclude.
\end{proof}

Now consider Laurent polynomials $P(n,M)$ in the variables $n$ and $M$.
We order monomials $n^x M^y$, with the lexicographic order on $(x+y,x)$;
this is consistent with the natural asymptotic ordering when $A_n$ holds.
\begin{lemma}
  \label{lem:AsympVar1}
  Assume that 
 \begin{equation}
  P(n,M)= n^{x_0} M^{y_0} + \text{smaller degree terms}
  \label{eq:PnM}
\end{equation}
  with $y_0 \ne 0$ (in particular, $M$ does appear in the dominant monomial).
  Then
  \[\Var(P(n,M))= \Var( n^{x_0} M^{y_0})(1+o(1)) = \tTheta(n^{2x_0+2y_0-1}).\]
\end{lemma}
\begin{proof}
As above, we write $M=m_n+\sigma_n Z_n$, with $Z_n$ asymptotically normal.
For any $x_1,y_1,x_2,y_2$ with $y_1 \ne y_2$, we have
\begin{multline*}
  \Cov(n^{x_1} M^{y_1} ,n^{x_2} M^{y_2}) =\\
  n^{x_1+x_2} \sum_{k_1,k_2 \ge 0} \binom{y_1}{k_1} \binom{y_2}{k_2} m_n^{y_1+y_2-k_1-k_2} \sigma_n^{k_1+k_2} \Cov( Z_n^{k_1},Z_n^{k_2}),
\end{multline*}
where the sum might be finite or infinite depending on the signs of $y_1$ and $y_2$.
The summand with the largest asymptotic behaviour correspond to $k_1=k_2=1$
(if $k_1=0$ or $k_2=0$,         
the corresponding summand is $0$).
Therefore, when the sum is finite (i.e. when $y_1,y_2>0$), we have
\[\Cov(n^{x_1} M^{y_1} ,n^{x_2} M^{y_2}) 
 = n^{x_1+x_2} m_n^{y_1+y_2-2} \sigma_n^2 (1+o(1)).\]
The same can be proved when $y_1$ and/or $y_2$ is negative,
using the technique of \cref{prop:Cumu_MInverse}; details are left to the reader.

Writing $P(n,M)$ as a sum of monomials and expanding $\Var(P(n,M))$ concludes the proof.
\end{proof}

We now give a similar result when $P(n,M)$ is perturbed by an error term with sufficiently small variance.
\begin{lemma}
  \label{lem:AsympVar2}
  We take a Laurent polynomial $P(n,M)$ as in \eqref{eq:PnM}
  and let $E(n,M)$ be a function with 
  $\Var(E(n,M))= \tO(n^{2x_0+2y_0-3})$.
   Then
  \[ \Var(P(n,M)+E(n,M)) = \Var(P(n,M))(1+o(1)) = \tTheta(n^{2x_0+2y_0-1}).\] 
\end{lemma}
\begin{proof}
  This is a trivial application of Cauchy-Schwartz inequality after expanding the variance.
\end{proof}
Looking at \cref{eq:CondEsper_BigSum,lem:FormCondExp,corol:Var}, we see that $ \esper(\OccSP|M)$
is of the form $P(n,M)+E(n,M)$, where $\Var(E(n,M))= \tO(n^{2\ell-2a-3})$
and $P(n,M)$ is a Laurent polynomial with dominant term $\tfrac{1}{(\ell-t)!} n^{\ell-t} M^{t-a}$.
If $t \ne a$, \cref{lem:AsympVar2} directly applies and \eqref{eq:VarToJustify} is proved.
If $t=a$ we simply apply \cref{lem:AsympVar2} to $\esper(\OccSP|M) - \tfrac{1}{(\ell-t)!} n^{\ell-t}$:
this difference has the same variance as $\esper(\OccSP|M)$ (we removed a deterministic quantity)
and is of the desired form: its dominant term is $\Theta(n^{\ell-t-1} M^{-1})$.
This concludes the proof of \eqref{eq:VarToJustify}.

\section*{Acknowledgements}
The author would like to thank Marko Thiel, with whom this work was initiated.
Many thanks are also due to an anonymous referee, whose suggestions greatly improved
the presentation of the paper.


\begin{thebibliography}{AAA{\etalchar{+}}01}

\bibitem[AAA{\etalchar{+}}01]{albert2001permutations}
M.~Albert, R.~Aldred, M.~Atkinson, C.~Handley, and D.~Holton.
\newblock Permutations of a multiset avoiding permutations of length 3.
\newblock {\em Eur. J. Comb.}, 22(8):1021--1031, 2001.

\bibitem[B{\'o}n10]{BonaMonotonePatterns}
M.~B{\'o}na.
\newblock {On three different notions of monotone subsequences}.
\newblock In {\em {Permutation Patterns}}, volume 376 of {\em {London Math.
  Soc. Lecture Note Series}}, pages 89--113. Cambridge University Press, 2010.

\bibitem[Bri69]{brillinger1969totalcumulance}
D.~Brillinger.
\newblock The calculation of cumulants via conditioning.
\newblock {\em Ann. Instit. Stat. Math.}, 21(1):215--218, 1969.

\bibitem[Can95]{canfield1995engel}
E.~R. Canfield.
\newblock {Engel's inequality for Bell numbers}.
\newblock {\em J. Combin. Th., Series A}, 72(1):184--187, 1995.

\bibitem[CDE18]{crane18Mallows}
H.~Crane, S.~DeSalvo, and S.~Elizalde.
\newblock The probability of avoiding consecutive patterns in the mallows
  distribution.
\newblock {\em Random Structures \& Algorithms}, 53(3):417--447, 2018.

\bibitem[CDKR14]{diaconis14average}
B.~Chern, P.~Diaconis, D.~Kane, and R.~Rhoades.
\newblock Closed expressions for averages of set partition statistics.
\newblock {\em Res. Math. Sci.}, 1:Art. 2, 32, 2014.

\bibitem[CDKR15]{CLT_SetPartitionsStatistics}
B.~Chern, P.~Diaconis, D.~Kane, and R.~Rhoades.
\newblock Central limit theorems for some set partition statistics.
\newblock {\em Adv. Appl. Math.}, 70:92--105, 2015.

\bibitem[Cha08]{chatterjee08new}
S.~Chatterjee.
\newblock A new method of normal approximation.
\newblock {\em Ann. Probab.}, 36(4):1584--1610, 2008.

\bibitem[CJZ11]{canfield11mahonian}
R.~Canfield, S.~Janson, and D.~Zeilberger.
\newblock {The Mahonian Probability Distribution on Words is Asymptotically
  Normal}.
\newblock {\em Adv. Appl. Math.}, 46:109--124, 2011.

\bibitem[F{\'e}r13]{feray13Ewens}
V.~F{\'e}ray.
\newblock Asymptotic behavior of some statistics in {E}wens random
  permutations.
\newblock {\em Electron. J. Probab.}, 18:no. 76, 32pp, 2013.

\bibitem[F{\'e}r18]{feray18dependency}
V.~F{\'e}ray.
\newblock {Weighted Dependency Graphs}.
\newblock {\em Electron. J. Probab.}, 23:no. 93, 65 pp, 2018.

\bibitem[FMN16]{ModGaussian1}
V.~F{\'e}ray, P.L. M{\'e}liot, and A.~Nikeghbali.
\newblock {\em {Mod-$\varphi$ convergence: Normality zones and precise
  deviations}}.
\newblock Springer, 2016.

\bibitem[FSV06]{flajolet06hidden}
P.~Flajolet, W.~Szpankowski, and B.~Vall\'{e}e.
\newblock Hidden word statistics.
\newblock {\em J. ACM}, 53(1):147--183, 2006.

\bibitem[Ful04]{fulman04inversions}
J.~Fulman.
\newblock Stein's method and non-reversible {M}arkov chains.
\newblock In {\em Stein's method: expository lectures and applications},
  volume~46 of {\em IMS Lecture Notes Monogr. Ser.}, pages 69--77. Inst. Math.
  Statist., Beachwood, OH, 2004.

\bibitem[Gol05]{goldstein05CombiCLT}
L.~Goldstein.
\newblock Berry-{E}sseen bounds for combinatorial central limit theorems and
  pattern occurrences, using zero and size biasing.
\newblock {\em J. Appl. Probab.}, 42(3):661--683, 2005.

\bibitem[Har67]{HarperCLTBlocks}
L.~H. Harper.
\newblock {Stirling behavior is asymptotically normal}.
\newblock {\em Ann. Math. Statist.}, 38(2):410--414, 1967.

\bibitem[Hoe48]{hoeffding48UStat}
W.~Hoeffding.
\newblock A class of statistics with asymptotically normal distribution.
\newblock {\em Ann. Math. Statistics}, 19:293--325, 1948.

\bibitem[Hoe63]{HoeffdingInequality}
W.~Hoeffding.
\newblock Probability inequalities for sums of bounded random variables.
\newblock {\em J. Amer. Stat. Assoc.}, 58(301):13--30, 1963.

\bibitem[Hof18]{hofer2018central}
L.~Hofer.
\newblock A central limit theorem for vincular permutation patterns.
\newblock {\em Disc. Math. Th. Comp. Science}, 19(2):paper \# 9, 26pp., 2018.

\bibitem[Jan88]{JansonDependencyGraphs}
S.~Janson.
\newblock {Normal convergence by higher semiinvariants with applications to
  sums of dependent random variables and random graphs}.
\newblock {\em Ann. Probab.}, 16(1):305--312, 1988.

\bibitem[Jan94]{JansonOrthogonalDecomposition}
S.~Janson.
\newblock {\em {Orthogonal decompositions and functional limit theorems for
  random graph statistics}}, volume 111.
\newblock Memoirs of Amer. Math. Soc., 1994.
\newblock 92 pp.

\bibitem[Jan97]{janson1997gaussian}
S.~Janson.
\newblock {\em Gaussian {H}ilbert Spaces}, volume 129 of {\em Cambridge Tracts
  in Mathematics}.
\newblock Cambridge University Press, 1997.

\bibitem[J{\L}R00]{JansonLuczakRucinski2000}
S.~Janson, T.~{\L}uczak, and A.~Ruci{\'n}ski.
\newblock {\em {Random graphs}}, volume~45 of {\em {Wiley Series in Discrete
  Mathematics and Optimization}}.
\newblock Wiley-Interscience, 2000.

\bibitem[JNZ15]{janson2015asymptotic}
S.~Janson, B.~Nakamura, and D.~Zeilberger.
\newblock On the asymptotic statistics of the number of occurrences of multiple
  permutation patterns.
\newblock {\em J. Combin.}, 6(1-2):117--143, 2015.

\bibitem[LS59]{LeonovShiryaevCumulants}
V.~P. Leonov and A.~N. Shiryaev.
\newblock {On a method of calculation of semi-invariants}.
\newblock {\em Th. Probab. Appl.}, 4:319--329, 1959.

\bibitem[Mik91]{MikhailovDependencyGraphs}
V.~G. Mikhailov.
\newblock {On a {Theorem} of {Janson}}.
\newblock {\em Th. Probab. Appl.}, 36(1):173--176, 1991.

\bibitem[MU05]{ProbAndComputing}
M.~Mitzenmacher and E.~Upfal.
\newblock {\em Probability and computing: Randomized algorithms and
  probabilistic analysis}.
\newblock Cambridge university press, 2005.

\bibitem[Rot64]{rota64partitions}
G.-C. Rota.
\newblock The number of partitions of a set.
\newblock {\em Amer. Math. Monthly}, 71:498--504, 1964.

\bibitem[Ruc88]{RucinskiCLTSubgraphs}
A.~Ruci{\'n}ski.
\newblock {When are small subgraphs of a random graph normally distributed?}
\newblock {\em Probab. Th. Rel. Fields}, 78(1):1--10, 1988.

\bibitem[Sch47]{Schutzenberger1947}
M.-P. Sch{\"u}tzenberger.
\newblock {Sur certains param{\`e}tres caract{\'e}ristiques des syst{\`e}mes
  d'{\'e}v{\'e}nements compatibles et d{\'e}pendants et leur application au
  calcul des cumulants de la r{\'e}p{\'e}tition}.
\newblock {\em C. R. Acad. Sci. Paris}, 225:277--278, 1947.

\bibitem[Sta83]{stam1983RandomSetPartition}
A.~J. Stam.
\newblock Generation of a random partition of a finite set by an urn model.
\newblock {\em J. Combin. Th., Series A}, 35(2):231--240, 1983.

\bibitem[Thi16]{thiel2016inversion}
M.~Thiel.
\newblock The inversion number and the major index are asymptotically jointly
  normally distributed on words.
\newblock {\em Combin. Probab. Comp.}, 25(3):470--483, 2016.

\bibitem[Urs27]{ursell_1927}
H.~D. Ursell.
\newblock {The evaluation of Gibbs' phase-integral for imperfect gases}.
\newblock {\em Math. Proc. Cambridge Phil. Soc.}, 23(6):685–697, 1927.

\end{thebibliography}
\newcommand{\etalchar}[1]{$^{#1}$}

\end{document}